\documentclass[11pt,a4paper]{article}
\usepackage{tikz,subfigure}
\usetikzlibrary{shapes.callouts,decorations.pathmorphing}
\usepackage[active]{srcltx}
\usepackage[margin=1in]{geometry}

\usepackage[font=scriptsize,bf]{caption}
\usepackage{algorithm}
\usepackage{algorithmic}
\usepackage{epsfig,amssymb,amsfonts,amsmath,amsthm}
\usepackage[numbers,sort&compress,sectionbib]{natbib}
\usepackage{epsfig,amssymb,amsfonts,amsmath}
\usepackage{amssymb,amsmath,amsfonts}
\usepackage{multirow}
\usepackage{xspace}
\usepackage{tabularx}
\usepackage{xcolor}

\newcommand{\remove}[1]{}

\newcommand{\x}{{\ensuremath{ \mathbf{x}}}}

\newcommand{\ce}{\mathrm{e}}

\renewcommand{\deg}{\operatorname{deg}}

\renewcommand{\leq}{\leqslant}
\renewcommand{\geq}{\geqslant}

\newcommand{\Oh}{\mathcal{O}}

\newcommand{\Hmax}[2]{\mathsf{H}_{\max}}

\newcommand{\thmref}[1]{Theorem~\ref{thm:#1}}

\newcommand{\lemref}[1]{Lemma~\ref{lem:#1}}

\newcommand{\figref}[1]{Figure~\ref{fig:#1}}

\newcommand{\secref}[1]{Section~\ref{sec:#1}}

\newcommand{\eq}[1]{equation~\eqref{eq:#1}}

\newcommand{\Pro}[1]{\mathbf{Pr} \left[\,#1\,\right]}

\newcommand{\Mid}{\,\middle\vert\,}

\newcommand{\Ex}[1]{\mathbf{E} \left[\,#1\,\right]}
\newcommand{\EX}[2]{\mathbf{E}_{#1} \left[\,#2\,\right]}

\newcommand{\VAR}[2]{\mathbf{Var}_{#1} \left[\,#2\,\right]}

\newcommand{\ind}[1]{\mathbf{1}\left(#1\right)}

\renewcommand{\tilde}{\widetilde}
\renewcommand{\hat}{\widehat}
\renewcommand{\bar}{\overline}
\renewcommand{\epsilon}{\varepsilon}

\newcommand{\minuslast}{{\star}}

\usepackage{setspace}

\newtheorem{thm}{Theorem}  

\newtheorem{lem}[thm]{Lemma}

\newtheorem{cor}[thm]{Corollary}

\newtheorem{rem}[thm]{Remark}

\newtheorem{pro}[thm]{Proposition}

\renewcommand{\tilde}{\widetilde}

\numberwithin{thm}{section}
\numberwithin{equation}{section}

\title{Balls into Bins via Local Search}

\author{Paul Bogdan\thanks{Department of Electrical and Computer Engineering,
        Carnegie Mellon University, Pittsburgh, PA, USA.\ \ Email:
        \hbox{pbogdan@ece.cmu.edu}.} \and
   Thomas Sauerwald\thanks{Max Planck Institute for Informatics, Saarbr\"ucken, Germany.\ \ Email:
        \hbox{sauerwal@mpi-inf.mpg.de}.} \and
   Alexandre Stauffer\thanks{Microsoft Research, Redmond WA, USA.\ \ Email:
        \hbox{alstauff@microsoft.com}.} \and
   He Sun\thanks{Max Planck Institute for Informatics, Saarbr\"ucken, Germany.\ \ Email:
        \hbox{hsun@mpi-inf.mpg.de}.}}

\begin{document}

\maketitle
\begin{abstract}
We propose a natural process for allocating $n$ balls into $n$ bins that are organized as the vertices of an undirected graph $G$.
Each ball first chooses a vertex $u$ in $G$ uniformly at random.
Then the ball performs a local search in $G$ starting from $u$ until it reaches a vertex with local minimum load, where the ball is finally placed on.
In our main result, we prove that this process
yields a maximum load of only $\Theta(\log \log n)$ on expander graphs.
In addition, we show that for $d$-dimensional grids the maximum load is $\Theta\Big( \big(\frac{\log n}{\log \log n}\big)^{\frac{1}{d+1}}\Big)$.
Finally, for almost regular graphs with minimum degree $\Omega(\log n)$, we prove that the maximum load is constant and also reveal a fundamental difference between random and arbitrary tie-breaking rules.
\end{abstract}


%
\section{Introduction}

It is well known that if each of $n$ balls is placed sequentially into one of $n$ bins chosen independently and uniformly at random,
then the highest loaded bin is likely to contain $\Theta\big(\frac{\log n}{\log\log n}\big)$ balls.
We call this process the $1$-choice process.
Alternatively, in the $d$-choice process,
each ball is allowed to choose $d$ bins independently and uniformly at random and is placed in the least loaded among the $d$ bins.
It was shown by \citet{ABKU99} and \citet{KLH96} that the maximum load reduces drastically
to $\Theta\big(\frac{\log\log n}{\log d}\big)$ in the $d$-choice process.
The constants omitted in the $\Theta$ are known~\cite{ABKU99}
and can be improved by considering more careful tie-breaking rules
as shown by \citet{V03}. \citet{BCSV06} extended these results to the case where the number of balls is larger than the number of bins.

In some applications, it is important to allow each ball to choose bins in a \emph{correlated} way.
For example, such correlations occur naturally in distributed systems,
where the bins represent processors that are interconnected as a graph and the balls represent tasks that need to be assigned to processors.
From a pratical point of view,
letting each task choose $d$ \emph{independent} random bins may be undesirable,
since the cost of accessing two bins which are far away in the graph may be higher than accessing two bins which are nearby.
Furthermore, in some contexts, tasks are actually created by the processors, which are then able to forward tasks to other processors to achieve a
more balanced load distribution.
In such settings, allocating balls close to the processor that created it
is certainly very desirable as it reduces the costs of probing the load of a processor and allocating the task.

We propose a very natural and simple process for allocating balls into bins that are interconnected as a graph.
We refer to this process as \emph{local search allocation}.
At each step, a ball is ``born'' in a bin chosen independently and uniformly at random, which we call the \emph{birthplace} of the ball.
Then, starting from its birthplace, the ball performs a \emph{local search} in the graph,
where in each step the ball moves to the adjacent bin with the smallest
load, provided that the load is strictly smaller than the load of the bin the ball is currently in.
Unless otherwise stated, we assume that ties are broken uniformly at random.
The local search ends when the ball visits the first vertex that is a \emph{local minimum}, which is
a vertex for which no neighbor has
a smaller load.
After that, the next ball is born and is allocated according to the procedure described above, and this process is repeated.
See Figure~\ref{fig:localsearch} for an illustration.

\tikzstyle{knoten}=[circle, color=black, inner sep=2pt, fill=black]
\tikzstyle{edge}=[line width=0.5pt]

\tikzstyle{knoten}=[circle, color=black, inner sep=2pt, fill=black]
\tikzstyle{edge}=[line width=0.5pt]

\newcommand{\ball}[1]{ \shade[ball color=yellow, opacity=0.5] (#1) circle (.25cm) }
\newcommand{\nball}[1]{ \shade[ball color=red] (#1) circle (.25cm) }
\newcommand{\gball}[1]{ \draw [color=black, line width=0.7pt, fill=white,dashed, dash pattern=on 1pt off 1pt] (#1) circle (.22cm) }

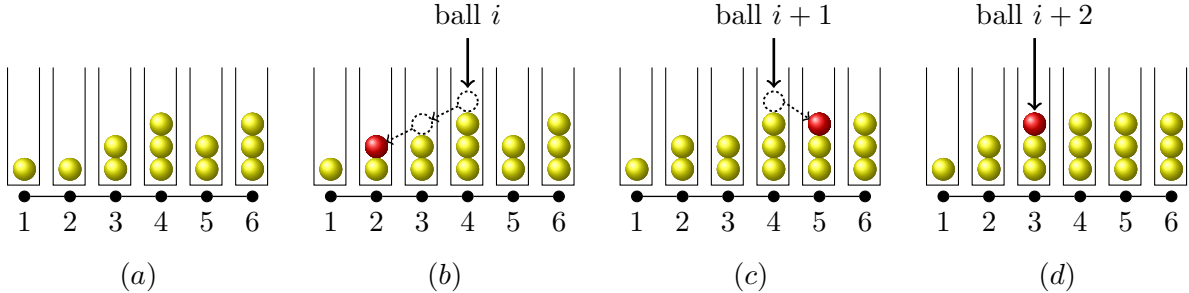
\begin{figure}[h]

\begin{tikzpicture}[auto, scale=0.6,  knoten/.style={
           draw=black, fill=black, thin, circle, inner sep=0.05cm}]
           \draw[white] (0.8,2) -- (7,5);
           \node[knoten] (1) at (1,2) [label=below:$1$]{};
           \node[knoten] (2) at (2,2) [label=below:$2$]{};
           \node[knoten] (3) at (3,2) [label=below:$3$]{};
           \node[knoten] (4) at (4,2) [label=below:$4$]{};
           \node[knoten] (5) at (5,2) [label=below:$5$]{};
           \node[knoten] (6) at (6,2) [label=below:$6$]{};
           \node[] (7) at (3.5,1) [label=below:$(a)$]{};
           \draw[edge] (1) -- (2) -- (3) -- (4) -- (5) -- (6);
           \draw [color=black] (1)++(0.33,2.85) to ++(0,-2.6)  to ++(-0.7,0) to ++ (0,2.6);
           \draw [color=black] (2)++(0.33,2.85) to ++(0,-2.6)  to ++(-0.7,0) to ++ (0,2.6);
           \draw [color=black] (3)++(0.33,2.85) to ++(0,-2.6)  to ++(-0.7,0) to ++ (0,2.6);
           \draw [color=black] (4)++(0.33,2.85) to ++(0,-2.6)  to ++(-0.7,0) to ++ (0,2.6);
           \draw [color=black] (5)++(0.33,2.85) to ++(0,-2.6)  to ++(-0.7,0) to ++ (0,2.6);
           \draw [color=black] (6)++(0.33,2.85) to ++(0,-2.6)  to ++(-0.7,0) to ++ (0,2.6);

           \ball{1,2.6};
           \ball{2,2.6};
           \ball{3,2.6};
           \ball{3,3.1};
           \ball{4,2.6};
           \ball{4,3.1};
           \ball{4,3.6};
           \ball{5,2.6};
           \ball{5,3.1};
           \ball{6,2.6};
           \ball{6,3.1};
           \ball{6,3.6};
\end{tikzpicture}
\begin{tikzpicture}[auto, scale=0.6,  knoten/.style={
           draw=black, fill=black, thin, circle, inner sep=0.05cm}]
                     \draw[white] (0.5,2) -- (7,5);
           \node[knoten] (1) at (1,2) [label=below:$1$]{};
           \node[knoten] (2) at (2,2) [label=below:$2$]{};
           \node[knoten] (3) at (3,2) [label=below:$3$]{};
           \node[knoten] (4) at (4,2) [label=below:$4$]{};
           \node[knoten] (5) at (5,2) [label=below:$5$]{};
           \node[knoten] (6) at (6,2) [label=below:$6$]{};
               \node[] (7) at (3.5,1) [label=below:$(b)$]{};
           \draw[edge] (1) -- (2) -- (3) -- (4) -- (5) -- (6);
           \draw [color=black] (1)++(0.33,2.85) to ++(0,-2.6)  to ++(-0.7,0) to ++ (0,2.6);
           \draw [color=black] (2)++(0.33,2.85) to ++(0,-2.6)  to ++(-0.7,0) to ++ (0,2.6);
           \draw [color=black] (3)++(0.33,2.85) to ++(0,-2.6)  to ++(-0.7,0) to ++ (0,2.6);
           \draw [color=black] (4)++(0.33,2.85) to ++(0,-2.6)  to ++(-0.7,0) to ++ (0,2.6);
           \draw [color=black] (5)++(0.33,2.85) to ++(0,-2.6)  to ++(-0.7,0) to ++ (0,2.6);
           \draw [color=black] (6)++(0.33,2.85) to ++(0,-2.6)  to ++(-0.7,0) to ++ (0,2.6);
           \draw[dashed, line width=0.7pt, ->,  dash pattern=on 1pt off 1pt] (4,4.1) -- (3.2,3.7);
           \draw[dashed, line width=0.7pt, ->,  dash pattern=on 1pt off 1pt] (3,3.6) -- (2.2,3.2);
           \draw[line width=1pt, ->] (4,5.5) -- (4,4.4);
           \node[] (7) at (4,6.7) [label=below:ball $i$]{};
           \ball{1,2.6};
           \ball{2,2.6};
           \nball{2,3.1};
           \ball{3,2.6};
           \ball{3,3.1};
           \gball{3,3.6};
           \ball{4,2.6};
           \ball{4,3.1};
           \ball{4,3.6};
           \gball{4,4.1};
           \ball{5,2.6};
           \ball{5,3.1};
           \ball{6,2.6};
           \ball{6,3.1};
           \ball{6,3.6};
\end{tikzpicture}
\begin{tikzpicture}[auto, scale=0.6,  knoten/.style={
           draw=black, fill=black, thin, circle, inner sep=0.05cm}]
                     \draw[white] (0.5,2) -- (7,5);
           \node[knoten] (1) at (1,2) [label=below:$1$]{};
           \node[knoten] (2) at (2,2) [label=below:$2$]{};
           \node[knoten] (3) at (3,2) [label=below:$3$]{};
           \node[knoten] (4) at (4,2) [label=below:$4$]{};
           \node[knoten] (5) at (5,2) [label=below:$5$]{};
           \node[knoten] (6) at (6,2) [label=below:$6$]{};
               \node[] (7) at (3.5,1) [label=below:$(c)$]{};
           \draw[edge] (1) -- (2) -- (3) -- (4) -- (5) -- (6);
           \draw [color=black] (1)++(0.33,2.85) to ++(0,-2.6)  to ++(-0.7,0) to ++ (0,2.6);
           \draw [color=black] (2)++(0.33,2.85) to ++(0,-2.6)  to ++(-0.7,0) to ++ (0,2.6);
           \draw [color=black] (3)++(0.33,2.85) to ++(0,-2.6)  to ++(-0.7,0) to ++ (0,2.6);
           \draw [color=black] (4)++(0.33,2.85) to ++(0,-2.6)  to ++(-0.7,0) to ++ (0,2.6);
           \draw [color=black] (5)++(0.33,2.85) to ++(0,-2.6)  to ++(-0.7,0) to ++ (0,2.6);
           \draw [color=black] (6)++(0.33,2.85) to ++(0,-2.6)  to ++(-0.7,0) to ++ (0,2.6);
           \draw[dashed, line width=0.7pt, ->,  dash pattern=on 1pt off 1pt] (4.2,4.1) -- (4.8,3.7);
           \draw[line width=1pt, ->] (4,5.5) -- (4,4.4);
           \node[] (7) at (4,6.7) [label=below:ball $i+1$]{};
           \ball{1,2.6};
           \ball{2,2.6};
           \ball{2,3.1};
           \ball{3,2.6};
           \ball{3,3.1};
           \ball{4,2.6};
           \ball{4,3.1};
           \ball{4,3.6};
           \gball{4,4.1};
           \ball{5,2.6};
           \ball{5,3.1};
           \nball{5,3.6};
           \ball{6,2.6};
           \ball{6,3.1};
           \ball{6,3.6};
\end{tikzpicture}
\begin{tikzpicture}[auto, scale=0.6,  knoten/.style={
           draw=black, fill=black, thin, circle, inner sep=0.05cm}]
                     \draw[white] (0.5,2) -- (7,5);
           \node[knoten] (1) at (1,2) [label=below:$1$]{};
           \node[knoten] (2) at (2,2) [label=below:$2$]{};
           \node[knoten] (3) at (3,2) [label=below:$3$]{};
           \node[knoten] (4) at (4,2) [label=below:$4$]{};
           \node[knoten] (5) at (5,2) [label=below:$5$]{};
           \node[knoten] (6) at (6,2) [label=below:$6$]{};
           \node[] (7) at (3.5,1) [label=below:$(d)$]{};
           \draw[edge] (1) -- (2) -- (3) -- (4) -- (5) -- (6);
           \draw [color=black] (1)++(0.33,2.85) to ++(0,-2.6)  to ++(-0.7,0) to ++ (0,2.6);
           \draw [color=black] (2)++(0.33,2.85) to ++(0,-2.6)  to ++(-0.7,0) to ++ (0,2.6);
           \draw [color=black] (3)++(0.33,2.85) to ++(0,-2.6)  to ++(-0.7,0) to ++ (0,2.6);
           \draw [color=black] (4)++(0.33,2.85) to ++(0,-2.6)  to ++(-0.7,0) to ++ (0,2.6);
           \draw [color=black] (5)++(0.33,2.85) to ++(0,-2.6)  to ++(-0.7,0) to ++ (0,2.6);
           \draw [color=black] (6)++(0.33,2.85) to ++(0,-2.6)  to ++(-0.7,0) to ++ (0,2.6);
           \draw[line width=1pt, ->] (3,5.5) -- (3,3.9);
           \node[] (7) at (3,6.7) [label=below:ball $i+2$]{};
           \ball{1,2.6};
           \ball{2,2.6};
           \ball{2,3.1};
           \ball{3,2.6};
           \ball{3,3.1};
           \ball{4,2.6};
           \ball{4,3.1};
           \ball{4,3.6};
           \ball{5,2.6};
           \ball{5,3.1};
           \ball{5,3.6};
           \nball{3,3.6};
           \ball{6,2.6};
           \ball{6,3.1};
           \ball{6,3.6};
\end{tikzpicture}

\caption{Illustration of the local search allocation. Black circles represent the vertices 1--6 arranged as a path, and the yellow circles represent the balls of the process (the most recently allocated ball is marked red). Figure~(a) shows the configuration after placing $i-1$ balls. As shown in Figure~(b), ball $i$ born at vertex $4$ has two choices in the first step of the local search (vertices~$3$ or~$5$) and is finally allocated to vertex~$2$. Figure~(c) and (d) shows the placement of ball $i+1$ and $i+2$.
}

  \label{fig:localsearch}

\end{figure}


In this paper we analyze the asymptotic behavior of the maximum load obtained when $n$ balls are allocated to $n$ bins as $n\to\infty$.
The main question is whether the local search allocation ensures a small load amount like the $d$-choice process with $d \geq 2$.
Our main result gives a positive answer for the important case where the $n$ balls are allocated to $n$ bins which are organized as the vertices of an \emph{expander} graph.
We show, in this case, that the maximum load is $\Theta(\log \log n)$, which has the same order of magnitude as of the $d$-choice process
(refer to Theorem~\ref{thm:expander} for a more general statement of \thmref{main} below).

\begin{thm}[Expander graphs]\label{thm:main}
   Let $G$ be any expander graph with constant maximum degree.
   Then, as $n\to\infty$, the maximum load after $n$ balls are allocated is $\Theta(\log \log n)$ with probability $1-o(1)$.
\end{thm}
\begin{rem}\label{rem:tiebreaking}{\rm
   Theorem~\ref{thm:main} holds also when ties are not broken uniformly at random but by means of a fixed permutation. In this
   tie-breaking procedure, for each vertex $v\in V$, we associate an arbitrary permutation $\xi_v$ of the neighbors of $v$.
   Then, whenever  a ball is currently at $v$ during its local search, the ball breaks ties by using the order in the permutation $\xi_v$.
   With this tie-breaking procedure, the local search allocation is a \emph{deterministic} function of the birthplaces of the balls and
   the permutations $\{\xi_v\}_{v\in V}$.
}\end{rem}

Before discussing and comparing our main result with existing results, we proceed to state our other results. Recall that, unless otherwise stated, we assume that in the
local search allocation ties are broken uniformly at random.
An important instance for applications is when the bins are organized as a ring or a grid.
The theorem below
establishes the maximum load in this case up to constant factors.
\begin{thm}[Grid graphs]\label{thm:torus}
   Let $G$ be any $d$-dimensional grid graph, where $d\geq 1$ is any integer constant.
   Then, as $n\to\infty$, the maximum load after $n$ balls are allocated is
   $\Theta\left( \Big(\frac{\log n}{\log \log n}\Big)^{\frac{1}{d+1}}\right)$ with probability $1-o(1)$.
\end{thm}
In comparison to \thmref{main}, the above theorem shows that the maximum load can be quite high on graphs with small expansion. Besides the expansion, it is also conceivable that a large degree ensures a small maximum load. The next theorem confirms this intuition by showing that a constant maximum load is obtained
for any almost regular graph with minimum degree $\Omega(\log n)$, where a graph is \emph{almost regular} if the ratio between the minimum and maximum degrees is a constant.
This class of graphs includes hypercubes and Erd\H{o}s-R\'enyi random graphs with average degree
$(1+\epsilon)\log n$, for any $\epsilon>0$.
\begin{thm}[Dense graphs]\label{thm:dense}
   Let $G$ be any almost regular graph with minimum degree $\Omega(\log n)$.
   Then, as $n\to\infty$, the maximum load after $n$ balls are allocated is $\Oh(1)$ with probability $1-o(1)$.
\end{thm}

In the next two theorems, we do not restrict the local search allocation to break ties uniformly at random.
In particular, we show that deviating from the usual random tie breaking rule can dramatically increase the maximum load.
Here, as in Remark~\ref{rem:tiebreaking}, we define a tie-breaking rule as a collection of variables
 $\{\xi_v\}_{v\in V}$, where each $\xi_v$ is a
permutation of the neighbors of $v$.
\begin{thm}[Impact of tie-breaking rules]\label{thm:tiebreaking}
   For any $d=\omega(1)$ as $n \rightarrow \infty$, there is a $d$-regular graph and a choice of $\{\xi_v\}_{v\in V}$
   for which the maximum load after $n$ balls are allocated is at least
   $\Omega\big( \min \big\{ d^{1/4}, \frac{\log n}{\log d} \big\}\big)$ with probability $1-o(1)$.
\end{thm}
To highlight the difference between random tie breaking and arbitrary tie breaking, Theorem~\ref{thm:tiebreaking} establishes that
there exists a $d$-regular graph $G$ with $d=\Theta(\log n)$ and a tie-breaking rule for which the
maximum load in $G$ is $\Omega((\log n)^{1/4})$. On the other hand, by Theorem~\ref{thm:dense} we have that breaking ties uniformly at
random leads to a constant maximum load in $G$.


Our final result establishes some lower bounds on the maximum load. These lower bounds hold for arbitrary tie-breaking rules.
\begin{thm}[Lower bounds]\label{thm:lowerbounds}
For any graph with maximum degree $\Delta$, the maximum load after $n$ balls are allocated is at least
   $\Omega\left( \frac{\log \log n}{\log \Delta} \right)$ with probability $1-o(1)$ as $n \rightarrow \infty$. Furthermore, for any integer $2 \leq d \leq \frac{\log n}{\ce}$, there is a $d$-regular graph for which the maximum load after $n$ balls are allocated is at least
   $\Omega \left(\sqrt{ \frac{\log n}{d \cdot \log \left(  \frac{\log n}{d} \right)}} \right)$ with probability $1-o(1)$ as $n \rightarrow \infty$.
\end{thm}
The combination of \thmref{tiebreaking} and the second statement of \thmref{lowerbounds}
shows that the two conditions concerning the degree and the random tie breaking in \thmref{dense} are not only sufficient but also necessary for obtaining a constant maximum load.

While in the $d$-choice process a ball only probes the load of up to $d$ bins,
the local search allocation may probe the load of $\omega(1)$ bins for some balls.
However, the number of bins whose load a given ball can probe is bounded above by $\Delta$ times the load of the birthplace of the ball, where $\Delta$
is the maximum degree of $G$.
Therefore,
the {\em expected} number of probed bins per ball is at most $\Delta$ and, for the case of expander graphs with constant maximum degree,
\thmref{main} implies that the maximum number of probed bins per ball is $\Oh(\log \log n)$.

An important feature of our local search allocation is that it maintains a smooth load distribution (cf.\ Lemma~\ref{lem:balanced});
i.e., the load difference for each edge of $G$ is at most one and balls are only placed in bins whose load is a local minimum. Hence, if each ball is controlled by an agent who strives for a minimization of their load, then the agents have a natural incentive to follow the local search allocation.

It is also important to remark that our process uses only a small amount of randomness.
For instance,
when ties are broken by means of a \emph{deterministic} collection of $\{\xi_v\}_{v\in V}$,
the only randomness comes from the birthplaces of the balls, which requires only
$n \log_2 n$ random bits.
From this perspective, it is comparable to the process by \citet{MPS02},
which use the same number of random bits and achieves a maximum load of $\Oh(\log \log n)$ as well.


{\bf Further Related Work.}
The work that is most related to ours is that of \citet{KP06}. They
analyzed a balls-into-bins model where each ball chooses a pair of adjacent bins uniformly at random
and is placed in the bin with the smaller load.
They proved that, for any $d$-regular graph, the maximum load after $n$ balls are allocated is $\log \log n + \Oh\big(1+ \frac{\log n}{\log (d/ \log^4 n)}\big)$,
and also showed an almost matching lower bound.
Hence, to retain the maximum load of $\Oh(\log \log n)$ from the 2-choice process,
the degree $d$ must be as large as $\Omega(n^{\Omega(1/(\log \log n))})$.
In contrast, our local search allocation achieves a maximum load of $\Oh(\log \log n)$ even for a large class of constant-degree graphs.
The model from \citet{KP06} was also studied implicitly in \citet{PTW10},
where the authors analyzed the {\em gap} defined as the difference between the maximum and average load.
They proved that, even if the number of balls $m$ is much larger than the number of bins $n$,
the gap is $\Theta(\log n)$ for expander graphs and, for the cycle, the gap is between $\Omega(\log n)$ and $\Oh(n \log n)$, i.e., independent of the number of balls.

In a different context, \citet{AHKV03} studied a related graph-based coupon collector process. In this process, each ball performs a local search
but distinguishes only between empty and nonempty bins. Therefore, each ball is allocated either to its birthplace or to one of its neighbors.
The authors analyzed the number of balls required until all bins are non-empty.

There are several other variations of balls-into-bins models for which the power of two choices has been analyzed
(we refer to~\cite{MRS01} for a survey).
For instance, \citet{BM05} considered a multidimensional version where balls correspond to $0/1$ vectors. \citet{B08}
analyzed balanced allocations on hypergraphs, extending the model of \cite{KP06} for allocations on graphs.
Among other results, the author obtained a constant maximum load for an almost regular hypergraph if each hyperedge consists of $d=\Theta(\log n)$ bins,
which is comparable to our result for dense graphs (\thmref{dense}).

Local search is a generic method to solve optimization problems and several recent studies in algorithmic game theory involve local search-based processes.
In contrast to our model, these processes usually start from a state where all tasks are assigned and allow, either sequentially or in parallel,
tasks to be reallocated by using so-called improvement steps (a.k.a.~selfish steps).
For various settings, lower and upper bounds on the number of improvement steps until a Nash equilibrium is found have been shown~\cite{EKM07,FGGHM07,G04}.
We note that the total number of improvement steps in our process is always bounded by $\Oh(n)$ with probability $1-o(1)$, regardless of the underlying graph.


%
%
%

\section{Basic properties}\label{sec:prelim}

We start with some notation. Let $G=(V,E)$ be an undirected, connected graph, where the $n$ vertices represent $n$ bins to which $n$ balls should be allocated. For each node $v\in V$, denote by $X_v^{(i)}$ the load of $v$ right after the $i$th ball is allocated.
Thus, we initially have $X_v^{(0)}=0$ for all $v\in V$.
Let $X_{\max}^{(n)} $ be the maximum load after $n$ balls have been allocated; i.e.,
$$
   X_{\max}^{(n)} = \max_{v\in V} X_v^{(n)}.
$$
Let $U_i \in V$ be the birthplace of ball $i$, so $U_i$ is a uniformly random sample from $V$. Recall that in the $1$-choice process, for all $i\geq 1$,
ball $i$ is allocated to $U_i$.
For any $v\in V$, let $\bar X_v^{(n)}$ be the load of $v$ after $n$ balls are allocated
according to the $1$-choice process. In symbols, we have
$$
   \bar X_v^{(n)} = |\{i\in [1,n] \colon U_i=v\}|.
$$
With this, define the maximum load for the $1$-choice process as
$$
   \bar X_{\max}^{(n)} :=\max_{v\in V} \bar X_v^{(n)}.
$$

Now, for two vectors $A=(a_1,a_2,\ldots,a_n)$ and $A'=(a_1',a_2',\ldots,a_n')$ such that $\sum_{i=1}^n a_i = \sum_{i=1}^n a_i'$,
we say that $A$ \emph{majorizes} $A'$ if, for each $\kappa=1,2,\ldots,n$, the sum of the $\kappa$ largest entries
of $A$ is at least the sum of the $\kappa$ largest  entries of $A'$. More formally,
if $j_1,j_2,\ldots,j_n$ are distinct numbers such that
$a_{j_1}\geq a_{j_2} \geq \cdots \geq a_{j_n}$ and
$j_1',j_2',\ldots,j_n'$ are distinct numbers  such that
$a'_{j_1'}\geq a'_{j_2'} \geq \cdots \geq a'_{j_n'}$, then
$$
   \text{$A$ majorizes $A'$ if }\sum_{i=1}^\kappa a_{j_i} \geq \sum_{i=1}^\kappa a'_{j_i'} \quad\text{for all $\kappa=1,2,\ldots,n$.}
$$
The lemma below establishes that the load vector obtained by the $1$-choice process majorizes the load vector obtained by the local search allocation. As a consequence, the maximum load of our local search allocation process is $\Oh( \log n / (\log \log n))$ for any graph.
\begin{lem}[Comparison with $1$-choice]\label{lem:1choice}
   For any fixed $k\geq 0$,
   we can couple $X^{(k)}$ and $\bar X^{(k)}$ so that, with probability 1, $\bar X^{(k)}$ majorizes $X^{(k)}$.
   Consequently, we have that, for all $k\geq 0$, $\bar X_{\max}^{(k)}$ stochastically dominates $X_{\max}^{(k)}$.
\end{lem}
\begin{proof}
   The proof is by induction on $k$. Clearly, for $k=0$, we have $X_v^{(0)}= \bar X_v^{(0)}=0$ for all $v\in V$.
   Now, assume that we can couple $X^{(k-1)}$ with $\bar X^{(k-1)}$ so that $\bar X^{(k-1)}$ majorizes $X^{(k-1)}$.
   Now let $j_1,j_2,\ldots, j_n$ be distinct elements of $V$ so that
   $X_{j_1}^{(k-1)} \geq X_{j_2}^{(k-1)} \geq \cdots \geq X_{j_n}^{(k-1)}$.
   Similarly, let $j_1',j_2',\ldots, j_n'$ be distinct elements of $V$ so that
   $\bar X_{j'_1}^{(k-1)} \geq \bar X_{j'_2}^{(k-1)} \geq \cdots \geq \bar X_{j'_n}^{(k-1)}$.
   Now let $\ell$ be a uniformly random integer from $1$ to $n$.
   Then, for the process $(X_v^{(k)})_{v\in V}$, we let the $k$th ball be born at vertex $j_\ell$ and
   define $\iota$ such that $j_\iota$ is the vertex to which the $k$th ball is allocated.
   Note that, $\iota \geq \ell$.
   For the process $(\bar X_v^{(k)})_{v\in V}$, we set the birthplace of
   the $k$th ball to $j_\ell'$. Therefore, for any $\kappa=1,2,\ldots,n$, we have
   $$
      \sum_{i=1}^\kappa  \bar X_{j'_i}^{(k)}
      = \sum_{i=1}^\kappa  \bar X_{j'_i}^{(k-1)} + \ind{\kappa \geq \ell}
      \geq \sum_{i=1}^\kappa  X_{j_i}^{(k-1)} + \ind{\kappa \geq \ell}
      \geq \sum_{i=1}^\kappa  X_{j_i}^{(k-1)} + \ind{\kappa \geq \iota}
      = \sum_{i=1}^\kappa  X_{j_i}^{(k)},
   $$
   where the first inequality follows by the induction hypothesis and the second inequality holds since $\iota \geq \ell$.
\end{proof}

For any $v\in V$, let $N_v$ be the set of neighbors of $v$ in $G$.
The next lemma establishes that the local search allocation always maintains a \emph{smoothed} load vector in the sense that the load of
any two adjacent vertices differs by at most $1$.
\begin{lem}[Smoothness]\label{lem:balanced}
   For any $k\geq 0$, any $v\in V$ and any $u\in N_v$, we have that $|X_v^{(k)}-X_{u}^{(k)}|\leq 1$.
\end{lem}
\begin{proof}
   In order to obtain a contradiction, suppose that $X_v^{(k)} \geq X_u^{(k)} +2$, and let $j$ be the last ball allocated to
   $v$. Then, we have that
   \[
     X_v^{(j-1)} = X_v^{(k)} - 1 \geq X_u^{(k)} + 1 \geq X_u^{(j-1)} + 1.
   \]
   Therefore, the moment the $j$th ball is born, vertex $v$ has at least one neighbor with load
   strictly smaller than $v$. Therefore, ball $j$ is not allocated to $v$, establishing a contradiction.
\end{proof}

For any vertex $v\in V$ and integer $r\geq 0$, let $N_v^r$ be the set of vertices of $G$ whose distance to $v$ is exactly $r$ (in particular, $N_v^0=\{v\}$), and let $B_v^r$ be the
set of vertices of $G$ whose distance to $v$ is at most $r$; then $B_v^r=\bigcup\nolimits_{i=0}^r N_v^i$.
Below we show that, if for a given $v\in V$ we have an upper bound for the number of balls allocated to $B_v^r$, then we obtain an upper bound for the
load of $v$.
\begin{lem}[Upper Bound]\label{lem:characterization}
   Let $v$ be an arbitrary vertex of $G$.
   Suppose that there exists an integer $r\geq 1$ and a positive $\Psi$, that may depend on $n$, such that
   the total number of balls allocated to the vertices of $B_v^r$ is at most
   $\Psi |B_v^r|$; i.e.,
   $
      \sum_{u\in B_v^r} X_u^{(n)} \leq \Psi |B_v^r|.
   $
   Then, we have that
   $$
      X_v^{(n)} \leq \Psi + \sum_{i=0}^r i \frac{|N_v^i|}{|B_v^r|}.
   $$
\end{lem}
\begin{proof}
   Note that, by Lemma~\ref{lem:balanced}, for any $u\in N_v^i$, we have $X_u^{(n)} \geq X_v^{(n)}-i$.
   Using this and the condition of the lemma, we obtain
   $$
      \Psi |B_v^r|
      \geq \sum_{u\in B_v^r} X_u^{(n)}
      \geq \sum_{i=0}^r (X_v^{(n)}-i) |N_v^i|
      = X_v^{(n)} |B_v^r| - \sum_{i=0}^r i |N_v^i|.
   $$
\end{proof}

Complementing the previous lemma, we now prove a lower bound on the maximum load that depends only on the number of
balls born at a subset of vertices and the cardinality of a small ball around that subset.

\begin{lem}[Lower Bound]\label{lem:characterizationlower}
   For any subset $S \subseteq V$, let $\Phi_{S}:=\sum_{i=1}^n \ind{ U_i \in S }$ be the number of balls born in $S$.
   Then, the maximum load $\beta:= X_{\max}^{(n)}$ satisfies the following inequality:
   $$
      \beta \cdot |B_{S}^{\beta} |\geq \Phi_{S},
   $$
   where $B_{S}^{\beta} := \bigcup_{s\in S} B_s^{\beta}$ is the set of vertices with distance at most $\beta$ from $S$.
\end{lem}
\begin{proof}
   If the maximum load is $\beta$, then every ball born at some vertex $u$ is allocated in $B_{u}^{\beta}$, and clearly
   the load of any vertex in $B_{S}^{\beta}$ is at most $\beta$. Combining these two insights yields
   \[
      \Phi_{S} \leq \sum_{u \in B_{S}^{\beta}} X_u^{(n)} \leq \beta \cdot |B_{S}^{\beta}|,
   \]
   and therefore the maximum load $\beta$ must satisfy
   $
     \beta \cdot |B_{S}^{\beta}| \geq \Phi_{S}.
   $
\end{proof}

For the next two lemmas, we need ties to be broken either uniformly at random or by means of a fixed permutation $\xi_v$ of the neighbors of $v$ for
each $v\in V$.
The next lemma establishes that the load vector $X^{(n)}$ satisfies a Lipschitz condition, which will turn out to be crucial in our
proofs.
\begin{lem}[Lipschitz property]\label{lem:lipschitz}
   Let $k\geq 1$ be fixed and $u_1,u_2,\ldots,u_k \in V$ be arbitrary.
   Let $(X^{(k)}_v)_{v\in V}$ be the load of the vertices of $G$ after the local search allocation places $k$ balls with birthplaces $u_1,u_2,\ldots,u_k$.
   Let $i\in \{1,2,\ldots,k\}$ be fixed, and let $(Y^{(k)}_v)_{v\in V}$ be the load of the vertices of $G$ after the local search allocation places
   $k$ balls
   with birthplaces $u_1,u_2,\ldots,u_{i-1},u_i',u_{i+1},u_{i+2},\ldots,u_k$, where $u_i'\in V$ is arbitrary.
   In other words, $Y_v^{(k)}$ is obtained from $X_v^{(k)}$ by changing the
   birthplace of the $i$th ball from $u_i$ to $u_i'$.
   Assume that, for both processes, the local search allocation breaks ties either uniformly at random
   or via the permutations $\{ \xi_v \}_{v\in V}$ described in Remark~\ref{rem:tiebreaking}.
   Then, there exists a coupling such that
   \begin{equation}
      \sum_{v\in V} |X_v^{(k)} - Y_v^{(k)}| \leq 2.
      \label{eq:lipschitz}
   \end{equation}
\end{lem}
\begin{proof}
   We refer to the process defining the variables $X^{(k)}$ as the $X$ process, and we refer to the
   process defining the variables $Y^{(k)}$ as the $Y$ process.
   If ties are broken uniformly at random, then for each $v\in V$ and $i\geq 1$, we define $\xi_v^{(i)}$ to be an independent and uniformly random permutation of
   the neighbors of $v$. We use this permutation for both the $X$ and $Y$ processes to break ties when ball $i$ is at vertex $v$.
   Then, since the first $i-1$ balls have the same birthplaces in both processes, we have that
   \begin{equation}
      X_v^{(i-1)} = Y_v^{(i-1)} \quad\text{ for all $v\in V$}.
      \label{eq:prefixunchanged}
   \end{equation}
   Now, when adding the $i$th ball, we let $v_i$ be the vertex to which this ball is allocated in the $X$ process and $v_i'$ be the vertex to which this
   ball is allocated in the $Y$ process.
   If $v_i=v_i'$, then $X_u^{(i)}=Y_u^{(i)}$ for all $u\in V$ and~\eqref{eq:lipschitz} holds.
   More generally, we have that
   \begin{equation}
       X_{v_i}^{(i)}=Y_{v_i}^{(i)}+\ind{v_i \neq v_{i}'},
       \quad Y_{v'_i}^{(i)}=X_{v'_i}^{(i)}+\ind{v_i \neq v_{i}'}
       \quad\text{and}\quad X_{u}^{(i)}=Y_{u}^{(i)} \text{ for $u \in V\setminus \{v_i,v_i'\}$.}
       \label{eq:desiredproperty}
   \end{equation}
   If $i=k$, then this implies~\eqref{eq:lipschitz} and the lemma holds.

   For the case $i<k$, we add ball $i+1$ and are going to define $v_{i+1}$ and $v_{i+1}'$
   so that~\eqref{eq:desiredproperty} holds with $i$ replaced by $i+1$. Then the proof of the lemma is completed by induction.
   We assume that $v_i\neq v_{i}'$, otherwise~\eqref{eq:lipschitz} clearly holds.
   We note that $v_{i+1}$ and $v_{i+1}'$ will not be in the same way as $v_i$ and $v_i'$. The role of $v_{i+1}$ and $v_{i+1}'$ is to be the only vertices whose loads in the $X$ and $Y$ processes are different.
   The definition of $v_{i+1}$ and $v_{i+1}'$ will vary depending on the situation.
   For this, let ball $i+1$ be born at $u_{i+1}$ and define $w$
   to be the vertex on which ball $i+1$ is allocated in the $X$ process and
   $w'$ to be the vertex on which ball $i+1$ is allocated in the $Y$ process.
   We can assume that $w\neq w'$, otherwise~\eqref{eq:desiredproperty} holds with $i$ replaced by $i+1$ by setting $v_{i+1}=v_i$ and $v_{i+1}'=v_i'$.

   Now we analyze ball $i+1$. It is crucial to note that, during the local search of ball $i+1$, if it does not enter $v_i$ in the $Y$ process
   and does not enter $v_i'$ in the $X$ process, then ball $i+1$ follows the same path in both processes. Since we are in the case $w\neq w'$,
   we can assume without loss of generality that ball $i+1$ eventually visits $v_i$ in the $Y$ process. In this case, since the local search performed by ball $i$
   in the $X$ process stops at vertex $v_i$, we have that $v_i$ is a local minimum for ball $i+1$ in process $Y$, which implies
   that $w'=v_i$.
   (The case when ball $i+1$ visits $v_i'$ in the $X$ process follows by a symmetric argument.)
   So, since $w\neq w'$, we have $X_{v_i}^{(i+1)}=Y_{v_i}^{(i+1)}$.
   Then we let $v_{i+1}=w$. If $w=v_i'$, we set $v_{i+1}'=w$ and~\eqref{eq:desiredproperty} holds since $X_u^{(i+1)}=Y_u^{(i+1)}$ for all $u\in V$.
   Otherwise we set $v_{i+1}'=v_i'$,
   and~\eqref{eq:desiredproperty} holds as well.
\end{proof}

The following is a consequence of Lemma~\ref{lem:lipschitz}.
\begin{lem}[Monotonicity]\label{lem:monotonicity}
   Let $k\geq 1$ be fixed and $u_1,u_2,\ldots,u_k \in V$ be arbitrary.
   Let $(X^{(k)}_v)_{v\in V}$ be the load of the vertices after $k$ balls are allocated with birthplaces $u_1,u_2,\ldots,u_k$.
   Let $i\in \{1,2,\ldots,k\}$ be fixed, and let $(Z^{(i,k)}_v)_{v\in V}$ be the load of the vertices of $G$ after $k-1$ balls are allocated
   with birthplaces $u_1,u_2,\ldots,u_{i-1},u_{i+1},u_{i+2},\ldots,u_k$.
   In other words, $Z_v^{(i,k)}$ is obtained from $X_v^{(k)}$ by removing ball $i$.
   Assume that, for both processes, the local search allocation breaks ties either uniformly at random
   or via the variables $(\xi_v)_{v\in V}$ described in Remark~\ref{rem:tiebreaking}.
   Then, there exists a coupling such that
   $$
      \sum_{v\in V} |X_v^{(k)} - Z_v^{(i,k)}| =1.
   $$
\end{lem}
\begin{proof}
   Let $G'$ be the graph obtained from $G$ by adding an isolated node $w$; i.e., $G'$ has vertex set $V\cup \{w\}$ and the same edge set as $G$.
   Applying \lemref{lipschitz} to $G'$ with the same choice of $u_1,\ldots,u_k\in V$ and with $u_{i}'=w$ gives
   \[
      \sum_{v \in V \cup \{w\} } \left| X_v^{(k)} - Y_v^{(k)} \right| = 2.
   \]
   Since $Y_{w}^{(k)} = 1$, $X_w^{(k)} = 0$ and $Z_{v}^{(i,k)} = Y_v^{(k)}$ for any $v \in V$, we conclude that
   \[
       \sum_{v \in V } \left| X_v^{(k)} - Z_v^{(i,k)} \right| = \sum_{v \in V } \left| X_v^{(k)} - Y_v^{(k)} \right| = 1.
   \]
\end{proof}

We now use Lemma~\ref{lem:monotonicity} to prove a type of subadditivity property. In a simpler statement, we show that, if for $k$ independent of copies of the local
search allocation with $m$ balls the maximum load is at most $x$, then the maximum load obtained after placing $km$ balls via local search allocation
is at most $kx$.
\begin{lem}[Subadditivity]\label{lem:subadditivity}
   For any $1 \leq z \leq n$, and any $x\geq 0$, it holds that
   $$
      \Pro{X_{\max}^{(n)} \geq \lceil n/z \rceil \cdot x} \leq \lceil n/z \rceil \cdot \Pro{X_{\max}^{(z)} \geq x}.
   $$
\end{lem}
\begin{proof}
   Let $U_1,U_2,\ldots,U_n$ be independent uniform random samples from $V$. Then, $X_{\max}^{(n)}$ is the maximum load after $n$ balls are added to $G$ with
   birthplaces $U_1,U_2,\ldots,U_n$.
   We define $k:=\lceil n/z \rceil$ independent copies of the local search allocation,
   where in the first copy we allocate $z$ balls according to the birthplaces $U_1,U_2,\ldots,U_{z}$, in the second copy we
   allocate $z $ balls according to the birthplaces $U_{z+ 1},U_{z+2},\ldots,U_{2z}$ and so on and so forth. Hence, in total we allocate $\lceil n/z \rceil \cdot z \geq n$ balls.
   Let  $M_1,M_2,\ldots,M_k$ be the maximum load of each copy, respectively, after $z$ balls
   are allocated. Then, we claim that $X_{\max}^{(n)} \leq X_{\max}^{(k \cdot z)} \leq \sum_{i=1}^k M_i$. This follows since, for each copy $i$,
   after allocating the $z$ balls,
   we can successively add more balls in such a way that all vertices have load exactly $M_i$ in copy $i$.
   Then, by taking the union of all copies, we obtain a balls-into-bins process with $\sum_{i=1}^k  M_i$ balls,
   $n$ of which have birthplaces $U_1,U_2,\ldots,U_n$.
   Then, by Lemma~\ref{lem:monotonicity}, the maximum load in this process, which is $\sum_{i=1}^k M_i$, is at least $X_{\max}^{(n)}$. Therefore,
   \begin{align*}
      \Pro{X_{\max}^{(n)} \geq \lceil n/z \rceil \cdot x}
      &\leq \Pro{\sum\nolimits_{i=1}^k M_i \geq \lceil n/z \rceil \cdot x}\\
      &\leq \Pro{\bigcup\nolimits_{i=1}^k \{ M_i \geq x\}}
      \leq k \,\Pro{X_{\max}^{(z)} \geq x}.
   \end{align*}
\end{proof}

%

\section{Expander graphs}\label{sec:expander}

In this section we give the proof of our main result, Theorem~\ref{thm:main}, which establishes an upper bound for the maximum load after $n$ balls
are allocated to the vertices of an expander graph.
In fact, we can prove this theorem in a more general setting.
As before, for each $u\in V$ and $r=1,2,\ldots$, we define $N_u^r$ to be the set of vertices of $V$ whose distance to $u$ is exactly $r$, and
$B_u^r$ to be the set of vertices of $V$ whose distance to $u$ is at most $r$; in symbols,
$$
   N_u^r = \{ v\in V \colon \text{graph distance between $u$ and $v$ is $r$}\}
   \quad\text{ and }\quad
   B_u^r = \bigcup_{i=0}^r N_u^i.
$$
 We say that $G$ has \emph{exponential growth} if there exists a constant $\phi>0$ so that
\begin{equation}
   |B_u^r| \geq \min \left\{ \exp(\phi r), \frac{n}{2} \right\} \text{ for all $u\in V$ and $r\geq 0$.}
   \label{eq:expgrowth}
\end{equation}
Note that any graph with exponential growth has a diameter of $\Oh(\log n)$.
Moreover, if $G$ is an expander, then it has exponential growth.
Therefore, Theorem~\ref{thm:main} follows from the theorem below.
\begin{thm}\label{thm:expander}
   If $G$ has the exponential growth property defined in~\eqref{eq:expgrowth} and bounded degrees,
   then there exists a positive constant $C$ so that, as $n\to\infty$,
   $$
      \Pro{X_{\max}^{(n)} \geq C \log \log n} \leq n^{-\omega(1)}.
   $$
\end{thm}

We devote the remainder of this section to prove Theorem~\ref{thm:expander}.
We start with a high level description of the proof.
We claim that, for any vertex $v$ and some properly chosen $r_0=\Oh(\log\log n)$,
\begin{equation}
   \Pro{\sum_{u\in B_v^{r_0}} X_u^{(n)} \geq C |B_v^{r_0}|} \leq n^{-\omega(1)};
   \label{eq:claim}
\end{equation}
i.e., with very high probability,
the number of balls allocated to the vertices of $B_v^{r_0}$ is
at most $C|B_v^{r_0}|$.
Having established~\eqref{eq:claim}, the proof follows immediately by applying Lemma~\ref{lem:characterization} and taking
the union bound over all $v$.
Now, in order to prove~\eqref{eq:claim},
we use that Lemma~\ref{lem:lipschitz} establishes that the load of the vertices satisfies a Lipschitz condition;
i.e., if the birthplace of one ball is changed, the load vector can only change in two vertices.
Therefore, we can apply Azuma's inequality to bound the probability that $\sum_{u\in B_v^{r_0}} X_u^{(n)} \geq C |B_v^{r_0}|$;
however, for $r_0=\Oh(\log\log n)$, the probability bound
obtained via Azuma's inequality is not small enough to take the union bound over $v\in V$ later.
Nevertheless, Azuma's inequality gives a small enough bound when applied to any radii larger than some $R\gg r_0$.
Then, the idea is to control the number of balls allocated to $B_v^{R-1}$ by using the bounds obtained for all radii $r\geq R$, and then apply an
inductive argument to finally establish~\eqref{eq:claim}.

The main intuition why this analysis works is because a ball can only be allocated inside $B_v^{r}$ if the ball is either born inside $B_v^{r}$ or
it is born in a vertex $u$ at distance $j$ to $B_v^{r}$ but whose load,
at the moment the ball is born, is at least $j$. This is true because, at each time a ball moves from a vertex $u$ to a vertex $u'\in N_u$
during the local search, the load of $u$ must be strictly larger than the load of $u'$. In other words, the ball traverses a
\emph{load decreasing path} from its birthplace to the vertex on which the ball is placed. Therefore, for a ball allocated in $B_v^{r}$,
the larger the distance between the birthplace of this ball and $B_v^r$, the smaller the number of possibilities for the birthplace of the ball since
these vertices must have a large load at the moment the ball is born. This, in a high-level description, gives that if we change the birthplace of a ball
to a uniformly random vertex, the load of the vertices inside $B_v^r$ does not change with high probability. This allows us to control the variance
of the Lipschitz condition and apply a more refined version of Azuma's inequality to move from radius $r$ to radius $r-1$ inductively until we reach
radius $r_0$. We remark that we actually need to control not only the number of balls allocated to $B_v^r$ for all $r\in[r_0,R]$, but also
the number of balls allocated to nodes in $B_v^r$ whose load is at least $\ell$ for all $r\in [r_0,R]$ and many values of $\ell$. We defer
the details for the rigorous argument below.

Now we proceed to the rigorous argument.
We start by showing that the load at any given vertex has an exponential tail.
\begin{lem}\label{lem:exptail}
   Let $v$ be any given vertex of $V$ and let $\Delta$ be the maximum degree of $G$. Then, for any $z\geq 8\ce\Delta$,
   $$
      \Pro{X_v^{(n)} \geq z} \leq 2\left(\frac{4\ce\Delta}{z}\right)^z.
   $$
\end{lem}
\begin{proof}
We start defining a sequence of vertices $w_0,w_1,\ldots$ and time steps $t_0 > t_1 > \cdots$ such that, for every $j\geq 1$, ball $t_j$ is born at vertex $w_j$
   and allocated to $w_{j-1}$.
   We start by setting $w_0=v$ and $t_0=n$.
   Inductively for $j \geq 1$, we let $t_j$ be the last ball allocated to $w_{j-1}$ before time $t_{j-1}$ and set $w_j$ to be the vertex at which ball $t_j$ is born.
   So, for $j=1$, $t_1$ is the last ball allocated to $w_0$ and $w_1$ is the vertex where ball $t_1$ is born. Note that, whenever $X_{v}^{(n)}\geq z$, if $w_1=w_0$, we know that
   $X_{w_1}^{(t_1)}\geq z-1$. On the other hand, if $w_1\neq w_0$, then we have that $X_{w_1}^{(t_1)}\geq z+d(w_0,w_1)-1$, where $d(u,v)$ is the graph distance between $u$ and $v$. More general, for all $\ell \geq 1$, we have that
   $$
      X_{w_\ell}^{(t_\ell)} \geq z+\sum_{j=1}^{\ell}(d(w_{j-1},w_j)-1).
   $$
   We continue this procedure until we find a value of $\ell$ such that $X_{w_\ell}^{(t_\ell)}=0$.
   Note that, for each $j$, we have $X_{w_j}^{(t_j)} - X_{w_{j-1}}^{(t_{j-1})}\geq -1$; consequently, we can have $X_{w_\ell}^{(t_\ell)}=0$ only for $\ell \geq z$.
   In order to obtain an upper bound for $\Pro{X_v^{(n)} \geq z}$, we apply the first-moment method over all possible sequences $(w_1,w_2,\ldots,w_\ell) \in V^{\ell}$ and
   $t_1 > t_2 > \cdots > t_\ell$, for every $\ell\geq 1$, such that $\sum_{j=1}^\ell (d(w_{j-1},w_j)-1) \leq -z$. With this, we have
   \begin{align}
      \lefteqn{\Pro{X_v^{(n)} \geq z}}\notag\\
      &\leq \sum_{\ell\geq z}\sum_{w_1,w_2,\ldots w_\ell \atop t_1 > t_2 > \cdots > t_\ell} \ind{\sum_{j=1}^\ell (d(w_{j-1},w_j)-1) \leq -z} \Pro{\bigcap\nolimits_{j=1}^\ell \{\text{ball $t_j$ is born at $w_j$}\}}\nonumber\\
      &\leq \sum_{\ell\geq z}\sum_{w_1,w_2,\ldots w_\ell \atop t_1 > t_2 > \cdots > t_\ell} \ind{\sum_{j=1}^\ell (d(w_{j-1},w_j)-1) \leq -z} \frac{1}{n^{\ell}}.
      \label{eq:summing}
   \end{align}
   Let $\lambda_i = d(w_{j-1},w_j)-1\in\{-1,0,1,\ldots\}$.
   For any fixed $\ell$, we can estimate the number of possible sequences by counting the number of possibilities to choose $\lambda_j$ and $t_j$
   so that $\sum_{j=1}^\ell \lambda_j \leq -z$, and then counting the number of possibilities to choose the $w_j$ accordingly.
   Clearly, there are at most $\binom{n}{\ell}$ ways to choose the $t_j$.
   Let $k\geq z$ be the number of values of $j$ for which
   $\lambda_j=-1$. Then, the other $\ell-k$ values of $\lambda_j$ are all non-negative and must sum to at most $k-z$.
   With this, we can bound above the number of
   choices for the $t_j$ and $\lambda_j$ by
   \begin{equation}
      \binom{n}{\ell} \sum_{k = z}^\ell \binom{\ell}{k} \binom{\ell-z}{k-z}
      \leq \left(\frac{\ce\cdot n}{\ell}\right)^\ell 2^{\ell-z} \sum_{k = z}^\ell \binom{\ell}{k}
      \leq \left(\frac{\ce\cdot n}{\ell}\right)^\ell 2^{2\ell-z}.
      \label{eq:nt}
   \end{equation}
   Once the $\lambda_j$ are fixed, the number of choices for the $w_j$ is at most
   \begin{equation}
      \prod_{j=1}^\ell \Delta^{\lambda_j+1}
      \leq \Delta^{\ell-z}.
      \label{eq:nw}
   \end{equation}
   Plugging the estimates in~\eqref{eq:nt} and~\eqref{eq:nw} into~\eqref{eq:summing}, we have
   $$
      \Pro{X_v^{(n)} \geq z}
      \leq \sum_{\ell\geq z} \left(\frac{\ce\cdot n}{\ell}\right)^\ell \cdot \frac{2^{2\ell-z} \Delta^{\ell-z}}{n^\ell}
      \leq \sum_{\ell\geq z} \left(\frac{4\ce\cdot \Delta }{\ell}\right)^\ell
      \leq \sum_{\ell\geq z} \left(\frac{4\ce\cdot \Delta }{z}\right)^\ell
      \leq 2 \left(\frac{4\ce\cdot\Delta }{z}\right)^z,
   $$
   where the last inequality uses the fact that $z\geq 8\ce\cdot\Delta$.
\end{proof}

Throughout the section, we fix an arbitrary vertex $v$ and bound the number of balls allocated to the vertices of $B_v^r$ for all $r_0 \leq r \leq R$, where
\begin{equation}
   r_0 := \min\Big\{r \colon |B_v^r| \geq \log^{10}n\Big\}
   \quad\text{and}\quad
   R := \min\Big\{r \colon |B_v^r| \geq \frac{n}{4\Delta}\Big\}.
   \label{eq:defr}
\end{equation}
We will consider the balls $(B_v^r)_{r\geq r_0}$, and will bound the number of vertices in
$B_v^r$ with load at least $\ell$ for all integers $\ell \in [\ell_0,\ell_1]$, where
\begin{equation}
   \ell_0 := 8\ce\cdot\Delta^2
   \quad\text{and}\quad
   \ell_1 := \frac{\log n}{4 \log (2\Delta)}.
   \label{eq:defell}
\end{equation}
Then, for all $r$ and $\ell$, define
$$
   \Lambda_{r,\ell} = \{u\in B_v^r \colon X_u^{(n)} \geq \ell\}.
$$
In order to control $\Lambda_{r,\ell}$, we will need to estimate
the probability that the $n$th ball changes the load of vertices in $B_v^r$. For this last value,
we need to control the load of the vertices after $n-1$ balls have arrived. Then, we define
$$
   \Lambda_{r,\ell}^\minuslast := \{u\in B_v^r \colon X_u^{(n-1)} \geq \ell\}.
$$
Note that $\Lambda_{r,\ell}^\minuslast\subseteq \Lambda_{r,\ell}$.
By Lemma~\ref{lem:exptail}, we have that, for $\ell\geq \ell_0$ and any $r\geq 1$,
\begin{equation}
   \Ex{|\cup_{k\geq 0}\Lambda_{r+k,\ell+k}|}
   \leq \sum_{k\geq 0} |B_v^{r+k}| 2\left(\frac{4\ce\cdot\Delta}{\ell+k}\right)^{\ell+k}
   \leq |B_v^r| \sum_{k\geq 0} \Delta^{k} 2\left(\frac{1}{2\Delta}\right)^{\ell+k}
   \leq 4(2\Delta)^{-\ell} |B_v^r|.
   \label{eq:mgexp}
\end{equation}
Next define
\[
\Lambda_{\ell}^\minuslast := \left\{ u \in V \colon X_{u}^{(n-1)} \geq \ell \right \},
\]
hence, $\Lambda_{\ell}^\minuslast = \bigcup_{r=1}^{\infty} \Lambda_{r,\ell}^\minuslast$.
and define the event
\begin{align}
   L_{R}^\minuslast = \bigcap_{\ell=\ell_0}^{6 \ell_1} \Big\{ |\Lambda_{\ell}^\minuslast| \leq \frac{n}{4 \Delta} \cdot (2 \Delta)^{-\ell} + \frac{\log^7 n}{\ell} \Big\}\enspace.
   \label{eq:defLR}
\end{align}

From now on let $\mathcal{F}_i$ be the $\sigma$-algebra induced by the configuration obtained after $i$ balls are placed. More formally,
if $U_1,U_2,\ldots,U_i$ are the birthplaces of the first $i$ balls and, for each $v\in V$ and $j=1,2,\ldots,i$, we define $\xi_v^{(j)}$ to be an
independent uniformly random permutation of the neighbors of $v$,
where $\{\xi_v^{(j)}\}_{v\in V}$ are the permutations used to break ties uniformly at random for ball $j$,
then $\mathcal{F}_i$ is the $\sigma$-algebra induced by $U_1,\ldots,U_i$ and $\{\xi_v^{(j)}\}_{v\in V, 1\leq j \leq i}$.
\begin{lem}\label{lem:larger}
   Let $v\in V$ be fixed, and let $R$ and $L^\minuslast_R$ be as defined in~\eqref{eq:defr} and~\eqref{eq:defLR}, respectively.
   Then, there exist $n_0$ so that, for all $n\geq n_0$, we have
   $$
      \Pro{L_R^\minuslast}\geq 1-2 n^{-\log^5 n}.
   $$
\end{lem}
\begin{proof}
Recall that by \lemref{1choice}, there is a coupling such that with probability $1$, $\bar X^{(n)}$ majorizes $ X^{(n)}$. Hence the claim follows directly by \lemref{thomasrepair} and a union bound over all $\ell$ with $\ell_0 \leq \ell \leq \ell_1$.
\end{proof}

For any $r$ with $1\leq r < R$ and any $\ell\geq \ell_0$,
in order to bound the number of vertices in $\Lambda_{r,\ell}$, we will look at the probability that the $n$-th ball affects $\Lambda_{r,\ell}$. In other words,
we control the probability that $\Lambda_{r,\ell}$ is different from $\Lambda_{r,\ell}^\minuslast$.
Note that it is only possible that $\Lambda_{r,\ell}\neq \Lambda_{r,\ell}^\minuslast$ if the $n$-th ball is born at a vertex
of $\Lambda_{r,\ell-1}^\minuslast$ or if it is born at a vertex of $N_v^{r+k}$ with load at least $\ell-1+k$ for some $k\geq 1$; we shall bound this last
set of vertices by $\Lambda_{r+k,\ell-1+k}$.
We define inductively for $r<R$
$$
   L_{r,\ell_0-1}^\minuslast = L_{r+1,\ell_1}^\minuslast
   \quad \text{ and } \quad
   L_{R-1,\ell_0-1}^\minuslast = L_R^\minuslast
$$
and, for $\ell\geq \ell_0$,
\begin{equation}
   L_{r,\ell}^\minuslast = L_{r,\ell-1}^\minuslast \cap \Big\{\big|\Lambda_{r,\ell}^\minuslast\big| \leq 8 |B_v^{r}|(2\Delta)^{-\ell}+\frac{\log^7 n}{\ell}\Big\}.
   \label{eq:defLr}
\end{equation}

The next lemma establishes that, with high probability, the last ball cannot affect the load of a small set of vertices.
\begin{lem}\label{lem:addend}
   Let $v$ be any fixed vertex, $r\geq 1$ and $\ell\geq \ell_0$. Then, there exists a positive constant $c=c(\Delta)$ such that
   $$
      \Pro{ \bigcup\nolimits_{u\in \Lambda_{r,\ell-1}^\minuslast} \{X_u^{(n-1)}\neq X_u^{(n)}\} \Mid L_{r,\ell-1}^\minuslast}
      \leq \frac{c|B_v^{r}|(2\Delta)^{-\ell}}{n} + \frac{3\log^7 n\log \log n}{n}.
   $$
\end{lem}
\begin{proof}
   Note that the $n$-th ball can only change the load of a vertex in $\Lambda_{r,\ell-1}^\minuslast$ if it is born at a vertex of $B_v^r$ of load at least $\ell-1$ or
   if it is born at a vertex $u\in N_v^{r+k}$ with $X_u^{(n-1)}\geq \ell-1+k$ for some $k\geq 1$.
   Fixing any realization for the birthplaces of the first $n-1$ balls,
   and thereby fixing the sets $\Lambda_{r,\ell}^\minuslast$ for all $r$ and $\ell$,
   we have the following upper bound
   $$
      \Pro{\bigcup\nolimits_{u\in \Lambda_{r,\ell-1}^\minuslast} \{X_u^{(n-1)}\neq X_u^{(n)}\} \Mid \mathcal{F}_{n-1}}
      \leq \frac{|\Lambda_{r,\ell-1}^\minuslast|}{n} + \frac{|\bigcup_{k=1}^\infty\Lambda_{r+k,\ell-1+k}^\minuslast|}{n},
   $$
   where the probability above is taken over the choice of $U_n$ only.
   Note that $L_{r,\ell-1}^\minuslast$ is measurable with respect to $\mathcal{F}_{n-1}$ since the birthplace of
   the $n$-th ball is independent of any event in $\mathcal{F}_{n-1}$. Then, for $\ell > \ell_0$, we obtain
   \begin{align*}
      &\Pro{\bigcup\nolimits_{u\in \Lambda_{r,\ell-1}^\minuslast} \{X_u^{(n-1)}\neq X_u^{(n)}\} \Mid L_{r,\ell-1}^\minuslast}\\
      &\leq \frac{8|B_v^{r}|(2\Delta)^{-\ell+1}+\frac{\log^7 n}{\ell-1}}{n} + \sum_{k=1}^{R-r} \frac{8|B_v^{r+k}|(2\Delta)^{-\ell+1-k}+\frac{\log^7 n}{\ell-1+k}}{n} + \frac{8|B_v^{R}|(2\Delta)^{-\ell-R+r} + \frac{\log^ 7 n}{\ell-1} }{n},
   \end{align*}
   where the last term comes from $L^\minuslast_R$. Using the bounds $|B_v^{r+k}|\leq |B_v^r|\Delta^k$ and $|B_v^{R}|\leq |B_v^r|\Delta^{R-r}$, we obtain
   \begin{align*}
      &\Pro{\bigcup\nolimits_{u\in \Lambda_{r,\ell}^\minuslast} \{X_u^{(n-1)}\neq X_u^{(n)}\} \Mid L_{r,\ell-1}^\minuslast}\\
      &\leq \frac{8|B_v^{r}|(2\Delta)^{-\ell+1}+\frac{\log^7 n}{\ell-1}}{n} + \sum_{k=1}^{R-r+1} \frac{8|B_v^{r}|\Delta^{k}(2\Delta)^{-\ell+1-k}}{n} + \sum_{k=0}^{R-r} \frac{\log^7 n}{(\ell-1+k) n}\\
      &\leq \frac{8|B_v^{r}|(2\Delta)^{-\ell+1}+\frac{\log^7 n}{\ell-1}}{n} + \frac{16|B_v^{r}|(2\Delta)^{-\ell+1}}{n} + \frac{2\log^7 n\log\log n}{n}\\
      &\leq \frac{24|B_v^{r}|(2\Delta)^{-\ell+1}}{n} + \frac{3\log^7 n\log \log n}{n}.
   \end{align*}
   For $\ell=\ell_0$, we simply bound $|\Lambda_{r,\ell_0-1}^\minuslast|$ by $|B_v^r|$, which gives that
   \begin{align*}
      &\Pro{\bigcup\nolimits_{u\in \Lambda_{r,\ell_0}^\minuslast} \{X_u^{(n-1)}\neq X_u^{(n)}\} \Mid L_{r,\ell_0}^\minuslast}\\
      &\leq \frac{|B_v^{r}|}{n} + \sum_{k=1}^{R-r+1} \frac{8|B_v^{r}|\Delta^{k}(2\Delta)^{-\ell_0+1-k}}{n} + \sum_{k=0}^{R-r} \frac{\log^7 n}{(\ell_0-1+k) n}\\
      &\leq \frac{|B_v^{r}|}{n}+\frac{16|B_v^{r}|(2\Delta)^{-\ell_0+1}}{n} + \frac{2\log^7 n\log\log n}{n}.
   \end{align*}
\end{proof}

Now, for any $i\geq 1$, $r\in [r_0,R]$ and $\ell\in[\ell_0,\ell_1]$, we define $\Upsilon_{r,\ell}^i\subseteq V^i$ as the set
\begin{equation}
   \Upsilon_{r,\ell}^i = \left\{(u_1,u_2,\ldots,u_i)\in V^i \colon \Pro{L_{r,\ell}^\minuslast \Mid \bigcap\nolimits_{j=1}^{j_0} \{U_j=u_j\}} \geq 1-\frac{1}{n^2} \text{ for all $j_0=1,2,\ldots,i$}\right\}.
   \label{eq:goodseq}
\end{equation}
Intuitively, for any given $r$, $\ell$ and $i$, the set $\Upsilon_{r,\ell}^i$ contains the \emph{good} birthplace for the first $i$ balls so that the
event $L_{r,\ell}^\minuslast$ is likely to occur, conditioning on any prefix of the birthplaces.

While \lemref{addend} considered the effect of the last ball $n$, the following lemma studies the effect of replacing the birthplace of ball $i$ by a randomly chosen bin.

\begin{lem}\label{lem:switch}
   Let $i$, $\ell$, $r$ and $v$ be fixed.
   Let $f\colon V^n \to \mathbb{Z}$ be an increasing function that depends only on $(X_{v'}^{(n)})_{v'\in \Lambda_{r,\ell}}$ and is $1$-Lipschitz.
   Let $W_0,W_1,\ldots,W_{n-i}$ and $\hat W_0$ be i.i.d.\ random variables chosen uniformly from $V$.
   Then, if $(u_1,u_2,\ldots,u_{i-1})\in \Upsilon_{r,\ell}^{i-1}$, we have
   \begin{align*}
      &\Pro{f(u_1,\ldots,u_{i-1},W_0,W_1,\ldots,W_{n-i})\neq f(u_1,\ldots,u_{i-1},\hat W_0,W_{1},\ldots,W_{n-i})}\\
      &\leq \frac{6c |B_v^r|(2\Delta)^{-\ell+1}}{n} + \frac{18\log^7 n\log\log n}{n}+ \frac{5}{n^2},
   \end{align*}
   where $c$ is the constant from Lemma~\ref{lem:addend}.
\end{lem}
\begin{proof}
   In this proof all the probabilities are taken conditional on $U_j=u_j$ for all $j=1,2,\ldots,i-1$, but we will omit this dependence from the notation.
   The idea is to relate the probability above to
   $$
      \Pro{f(u_1,\ldots,u_{i-1},W_{1},\ldots,W_{n-i},W_0)\neq f(u_1,\ldots,u_{i-1},W_{1},\ldots,W_{n-i},\hat W_0)},
   $$
   which corresponds to changing the $n$-th ball instead of the $i$-th ball; this will allow us to apply Lemma~\ref{lem:addend}.

   Consider the three events below:
   \begin{align*}
      E_1 &= \big\{f(u_1,\ldots,u_{i-1},W_0,W_{1},\ldots,W_{n-i})\neq f(u_1,\ldots,u_{i-1},W_{1},\ldots,W_{n-i},W_0)\big\}\\
            \hat E_1&= \big\{f(u_1,\ldots,u_{i-1},\hat W_0,W_{1},\ldots,W_{n-i})\neq f(u_1,\ldots,u_{i-1},W_{1},\ldots,W_{n-i},\hat W_0)\big\} \\
                  E_\mathrm{end} &= \big\{f(u_1,\ldots,u_{i-1},W_{1},\ldots,W_{n-i},W_0)\neq f(u_1,\ldots,u_{i-1},W_{1},\ldots,W_{n-i},\hat W_0)\big\}.
   \end{align*}
   Clearly, $f(u_1,\ldots,u_{i-1},W_0,W_1,\ldots,W_{n-i})$ and $f(u_1,\ldots,u_{i-1},\hat W_0,W_{1},\ldots,W_{n-i})$ can only be different if
   at least one of $E_1$, $\hat E_1$ or $E_\mathrm{end}$ happen. Therefore,
   we can write
   \begin{align}
      &\Pro{f(u_1,\ldots,u_{i-1},W_0,W_1,\ldots,W_{n-i})\neq f(u_1,\ldots,u_{i-1},\hat W_0,W_{1},\ldots,W_{n-i})}\nonumber\\
      &\leq \Pro{E_1} + \Pro{\hat E_1} + \Pro{E_\mathrm{end}}.
      \label{eq:decomp}
   \end{align}

   We start with the term $\Pro{E_\mathrm{end}}$.
   Let $I_\mathrm{end}$ be the event that $L_{r,\ell-1}^\minuslast$ happens given that the birthplaces of the first $n-1$ balls are according to the sequence
   $(u_1,\ldots,u_{i-1},W_{1},\ldots,W_{n-i})$.
   If $Y_u^{(n)}$ is the load of vertex $u$ when $n$ balls are added with birthplaces
   $u_1$, $u_2$,~$\ldots$, $u_{i-1}$, $W_1$, $W_2$,~$\ldots$, $W_{n-i}$, $W_0$ and
   $\hat Y_u^{(n)}$ is the load of vertex $u$ when $n$ balls are added with birthplaces
   $u_1$, $u_2$,~$\ldots$, $u_{i-1}$, $W_1$, $W_2$,~$\ldots$, $W_{n-i}$, $\hat W_0$, we have that
   \begin{align}
      \Pro{E_\mathrm{end}}
      &\leq \Pro{I_\mathrm{end}^\mathrm{c}} +
         \Pro{\bigcup_{u\in B_v^r}\left( \{Y_u^{(n-1)}\neq Y_u^{(n)}\} \cap \{Y_u^{(n)} \geq \ell\}\right) \cap I_\mathrm{end}}\nonumber\\
      &\quad+ \Pro{\bigcup_{u\in B_v^r} \left(\{\hat Y_u^{(n-1)}\neq \hat Y_u^{(n)}\} \cap \{\hat Y_u^{(n)} \geq \ell\}\right) \cap I_\mathrm{end}},
      \label{eq:eend}
   \end{align}
   where $\Pro{I_\mathrm{end}^\mathrm{c}} \leq n^{-2}$ by~\eqref{eq:goodseq} since $(u_1,\ldots,u_{i-1})\in\Upsilon_{r,\ell}^{i-1}$ and $W_1,\ldots,W_{n-i}$ are
   i.i.d.\ uniform samples from $V$. The other two terms in \eq{eend} can be bounded by Lemma~\ref{lem:addend}.

   Now it remains to bound $\Pro{E_1}$, since by symmetry we have $\Pro{E_1}=\Pro{\hat E_1}$.
   In order to bound $\Pro{E_1}$,
   we consider all cyclic permutations of $(u_1,\ldots,u_{i-1},W_0,W_1,\ldots,W_{n-i})$. More specifically, we first compare
   $$
      f(u_1,\ldots,u_{i-1},W_0,W_1,\ldots,W_{n-i})
      \quad\text{with}\quad
      f(u_1,\ldots,u_{i-1},W_1,W_{2},\ldots,W_{n-i},W_0),
   $$
   then we compare
   $$
      f(u_1,\ldots,u_{i-1},W_1,W_{2},\ldots,W_{n-i},W_0)
      \quad\text{with}\quad
      f(u_1,\ldots,u_{i-1},W_2,W_3,\ldots,W_{n-i},W_0,W_1),
   $$
   and so on and so forth until
   we compare
   $$
      f(u_1,\ldots,u_{i-1},W_{n-i},W_0,W_1,\ldots,W_{n-i-1})
      \quad\text{with}\quad
      f(u_1,\ldots,u_{i-1},W_0,W_1,\ldots,W_{n-i}).
   $$
   In order to do this, we define a graph $H$ whose vertex set is
   $\{0,1,\ldots,n-i\}\times V^{n-i+1}$; so the vertices of $H$ have the form $(j,z_0,z_1,\ldots,z_{n-i})$.
   We let each vertex of $H$ have exactly one outgoing edge and one incoming edge by having a directed edge from each vertex
   $(j,z_0,z_1,\ldots,z_{n-i})$ to $(j',z_1,z_2,\ldots,z_{n-i},z_0)$ where $j'=j+1 \pmod{n-i+1}$. With this, the sequence
   \begin{align*}
      (0,z_0,z_1,\ldots,z_{n-i}),
      (1,z_1,\ldots,z_{n-i},z_0),
      \ldots,
      (j,z_j,\ldots,z_{n-i},z_0,z_1,\ldots,z_{j-1}),
      \ldots,\\
      (n-i,z_{n-i},z_0,\ldots,z_{n-i-1}),
      (0,z_0,z_1,\ldots,z_{n-i})
   \end{align*}
   forms a directed cycle in $H$.
   Since each vertex has in-degree and out-degree equal to one, we have that each connected component of $H$ is a directed cycle with $n-i+1$ vertices.
   Now, for any vertex $(j,z_0,z_1,\ldots,z_{n-i})$ of $H$, we say that
   the edge from $(j,z_0,z_1,\ldots,z_{n-i})$ to $(j+1,z_1,z_2,\ldots,z_{n-i},z_0)$ is \emph{$f$-increasing}
    if
   $$
      f(u_1,u_2,\ldots,u_{i-1},z_0,z_1,\ldots,z_{n-i})<f(u_1,u_2,\ldots,u_{i-1},z_1,z_2,\ldots,z_{n-i},z_0).
   $$
   Similarly we say that the edge is \emph{$f$-decreasing} if the opposite inequality holds:
   $$
      f(u_1,u_2,\ldots,u_{i-1},z_0,z_1,\ldots,z_{n-i})>f(u_1,u_2,\ldots,u_{i-1},z_1,z_2,\ldots,z_{n-i},z_0).
   $$

   Since $f$ is $1$-Lipschitz and integer-valued, for any directed cycle, the number of $f$-increasing edges is the same as the number of $f$-decreasing edges, and therefore the
   number of $f$-increasing edges in $H$ is the same as the number of $f$-decreasing edges.

   Note that choosing a sequence $(W_0,W_1,\ldots,W_{n-i})$ uniformly at random from
   $V^{n-i+1}$ and a number $k$ uniformly at random from
   $\{0,1,\ldots,n-i\}$ gives a uniformly random vertex of $H$. Similarly, if we choose a vertex $(j,z_0,z_1,\ldots,z_{n-i})$ uniformly at random
   from the vertex set of $H$, then $(z_0,z_1,\ldots,z_{n-i})$ is a uniform sample from $V^{n-i+1}$. Since each vertex of $H$ has a unique outgoing edge, one can choose
   a uniformly random edge of $H$ by choosing the outgoing edge of an uniformly random vertex of $H$.

   We fix a uniformly random vertex $(j,z_0,z_1,\ldots,z_{n-i})$ of $H$ and let $e$ be its unique outgoing edge; i.e., $e$ is the edge
   from $(j,z_0,z_1,\ldots,z_{n-i})$ to $(j+1,z_1,\ldots,z_{n-i},z_0)$.
   Therefore,
   \begin{equation}
      \Pro{E_1} \leq \Pro{\text{$e$ is $f$-increasing}} + \Pro{\text{$e$ is $f$-decreasing}} = 2\,\Pro{\text{$e$ is $f$-increasing}}.
      \label{eq:inconly}
   \end{equation}
   Let $Y_u^{(k)}$ be the load of vertex $u$ after $k$ balls are added with birthplaces given by the sequence
   $(u_1,\ldots,u_{i-1},z_0,\ldots,z_{n-i})$ and let
   $\hat Y_u^{(k)}$ be the load of vertex $u$ after $k$ balls are added with birthplaces given by the sequence
   $(u_1,\ldots,u_{i-1},z_1,\ldots,z_{n-i},z_0)$.

   Note that, by Lemma~\ref{lem:monotonicity}, removing a ball cannot increase the load of any vertex; this gives that
   $$
      Y_u^{(n)} \geq \hat Y_u^{(n-1)} \quad \text{for all $u\in V$}.
   $$
   Therefore, $e$ can only be $f$-increasing if $\hat Y_u^{(n)} \neq \hat Y_u^{(n-1)}$ and $\hat Y_u^{(n)}\geq \ell$ for some $u\in B_v^r$.
   Let $I_e$ be the indicator that $L_{r,\ell-1}^\minuslast$ holds given that the birthplaces of the balls are according to the sequence
   $(u_1,\ldots,u_{i-1},z_1,\ldots,z_{n-i},z_0)$.
   Hence, using~\eqref{eq:inconly}, we have
   $$
      \Pro{E_1} \leq 2\,\Pro{\bigcup\nolimits_{u\in B_v^r}\left( \{\hat Y_u^{(n-1)}\neq \hat Y_u^{(n)}\} \cap \{\hat Y_u^{(n)} \geq \ell\}\right)  \, \mid \, I_e} + 2\,\Pro{I_e^\mathrm{c}}.
   $$
   By~\eq{goodseq}, $\Pro{I_e^\mathrm{c}} \leq n^{-2}$ and the other term can be bounded by
   Lemma~\ref{lem:addend}.
   Then, we put this and~\eqref{eq:eend} into \eqref{eq:decomp} to complete the proof.
\end{proof}

Now we will use an inductive argument to bound the probability that $L_{r,\ell}$ happens for all $r$ and $\ell$.
\begin{lem}\label{lem:lambda}
   There exists a positive constant $c'$ such that, for any $v\in V$, $r<R$ and any integer $\ell\in [\ell_0,\ell_1]$, we have
   $$
      \Pro{|\Lambda_{r,\ell}|> 8|B_v^r|(2\Delta)^{-\ell} + \frac{\log^7 n}{\ell}}
      \leq \exp(-\log^4n) + n^2 \cdot \Pro{(L_{r,\ell-1}^\minuslast)^\mathrm{c}}
   $$
   for all large enough $n$.
\end{lem}
\begin{proof}
   Recall that, by Lemma~\ref{lem:exptail}, we have
   $$
      \Ex{| \Lambda_{r,\ell}|}
      \leq 2|B_v^{r}| \left(\frac{4\ce\cdot\Delta}{\ell}\right)^{\ell} \leq 4|B_v^r|\left(2\Delta\right)^{-\ell}.
   $$
   Define the stopping time
   $$
      \tau = n \land \min\{j \colon (U_1,U_2,\ldots,U_j) \not \in \Upsilon_{r,\ell}^j\},
   $$
   where for any two numbers $a \land b := \min \{a,b\}$.
   Let $X_i = \Ex{ | \Lambda_{r,\ell} | \mid \mathcal{F}_i}$ and $Z_i = X_{i \land \tau}$.
   Note that $Z_i$ is a martingale since $\tau$ is a stopping time.
   Moreover, $\tau$ is a \emph{bounded} stopping time and, using the optional stopping theorem, we have that
   $\Ex{Z_\tau}=\Ex{Z_0}=\Ex{X_0}$. We want to bound the conditional variance of $Z_i-Z_{i-1}$ uniformly over all $i$, which is given by
   \begin{align*}
      \VAR{U_i}{Z_i-Z_{i-1} \Mid \mathcal{F}_{i-1}} &= \EX{U_i}{(Z_i - Z_{i-1})^2 \Mid \mathcal{F}_{i-1}} - \left( \EX{U_i}{Z_i - Z_{i-1} \Mid \mathcal{F}_{i-1}} \right)^2 \\ &= \EX{U_i}{(Z_i - Z_{i-1})^2 \Mid \mathcal{F}_{i-1}},
   \end{align*}
   where the variance and expectation are taken over the choice of $U_i$ only.
   Now, we write
   \begin{align*}
      &\EX{U_i}{(Z_i - Z_{i-1})^2 \Mid \mathcal{F}_{i-1}}\\
      &= \EX{U_i}{(Z_i - Z_{i-1})^2\ind{\tau \geq i} \Mid \mathcal{F}_{i-1}}
         + \EX{U_i}{(Z_i - Z_{i-1})^2\ind{\tau < i} \Mid \mathcal{F}_{i-1}}\\
      &= \EX{U_i}{(Z_i - Z_{i-1})^2\ind{\tau \geq i} \Mid \mathcal{F}_{i-1}},
   \end{align*}
   since $Z_i=Z_{i-1}$ whenever $\tau< i$.
   We want to bound
   $$
      \EX{U_i}{(Z_i - Z_{i-1})^2 \Mid \bigcap\nolimits_{j=1}^{i-1} \{U_j=u_j\}},
   $$
   uniformly over all $(u_1,u_2,\ldots,u_{i-1})\in \Upsilon_{r,\ell}^{i-1}$.
   Let $\zeta_u$ be the value of $Z_i$ when $U_i=u$ and let $\bar \zeta = \frac{1}{n}\sum_{u\in V}\zeta_u$. Then we have
   $$
      \EX{U_i}{(Z_i - Z_{i-1})^2 \Mid \bigcap\nolimits_{j=1}^{i-1} \{U_j=u_j\}}
      = \frac{1}{n} \sum_{u \in V} (\zeta_u-\bar\zeta)^2.
   $$
   Since $X_i$ is 1-Lipschitz we have $|\zeta_u-\zeta_{u'}|\leq 1$ for all $u,u' \in V$.
   With this, we can write
   $$
      \frac{1}{n} \sum_{u \in V} (\zeta_u-\bar\zeta)^2
      \leq \frac{1}{n} \sum_{u\in V} |\zeta_u-\bar\zeta|
      = \frac{1}{n} \sum_{u\in V}\Big|\sum_{u'\in V}\frac{1}{n}(\zeta_{u} - \zeta_{u'})\Big|
      \leq \frac{1}{n^2} \sum_{u\in V}\sum_{u'\in V}\left|\zeta_{u} - \zeta_{u'}\right|.
   $$
   Now, using Lemma~\ref{lem:switch}, we have that
   $$
      \frac{1}{n^2} \sum_{u\in V}\sum_{u'\in V}\left|\zeta_{u} - \zeta_{u'}\right|
      \leq \frac{6c|B_v^r|(2\Delta)^{-\ell+1}}{n}+\frac{18\log^7 n \log\log n}{n} + \frac{5}{n^2},
   $$
   which gives that
   $$
      \EX{U_i}{(Z_i - Z_{i-1})^2 \Mid \mathcal{F}_{i-1}}
      \leq \frac{6c|B_v^r|(2\Delta)^{-\ell+1}}{n}+\frac{18\log^7 n \log\log n}{n} + \frac{5}{n^2},
   $$
   uniformly over $i$.
   Now, note that, for any $\lambda>0$,
   \begin{align*}
      \Pro{|X_n-X_0|>\lambda}
      &\leq \Pro{\{|X_n-X_0|>\lambda\}\cap \{\tau \geq n\} } + \Pro{\tau < n}\\
      &= \Pro{\{|Z_n-Z_0|>\lambda\}\cap \{\tau \geq n\} } + \Pro{\tau < n}.
   \end{align*}
    Also, we have that for any $1 \leq i \leq n$,
    \begin{align*}
       &\Pro{(L_{r,\ell-1}^\minuslast)^\mathrm{c}}  \\
       &= \sum_{j_0=1}^{i} \Pro{  (L_{r,\ell-1})^{c} \cap \left\{ (U_1,\ldots,U_{j_0}) \not\in \Upsilon_{r, \ell-1}^{j_0} \right\} \cap \left( \bigcap\nolimits_{k=1}^{j_0-1} \{(U_1,\ldots,U_{k}) \in \Upsilon_{r, \ell-1}^{k}\} \right) } \\
       & \quad\, + \Pro{ (L_{r,\ell-1}^\minuslast)^\mathrm{c} \cap \left( \bigcap\nolimits_{k=1}^i \{ (U_1,\ldots,U_{k}) \in \Upsilon_{r, \ell-1}^{k}   \} \right)   } \\
       &\geq \Pro{(U_1,\ldots,U_i) \not\in \Upsilon_{r,\ell-1}^i} \cdot \frac{1}{n^2}.
    \end{align*}
   This gives that
   $$
      \Pro{\tau < n}
      \leq  \Pro{(U_1,\ldots,U_{n-1}) \not \in \Upsilon_{r,\ell-1}^{n-1}}
      \leq n^2\cdot \Pro{(L_{r,\ell-1}^\minuslast)^\mathrm{c}}.
   $$
   Then, applying the version of Azuma's inequality from Lemma~\ref{lem:azumavar}, we have
   \begin{align*}
      &\Pro{|\Lambda_{r,\ell}|> 8|B_v^r|(2\Delta)^{-\ell} + \frac{\log^7 n}{\ell}}\\
      &\leq \Pro{|Z_n-Z_0|> 4|B_v^r|(2\Delta)^{-\ell} + \frac{\log^7 n}{\ell}} + n^2 \cdot \Pro{(L_{r,\ell-1}^\minuslast)^\mathrm{c}}\\
      &\leq \exp\left(-\frac{(4|B_v^r|(2\Delta)^{-\ell}+\frac{\log^7 n}{\ell})^2}{14c|B_v^r|(2\Delta)^{-\ell}+37\log^7 n\log\log n}\right) + n^2 \cdot \Pro{(L_{r,\ell-1}^\minuslast)^\mathrm{c}}.
   \end{align*}
   If $|B_v^r|(2\Delta)^{-\ell} \geq \log^7n\log\log n$, the exponential term above is at most
   $$
      \exp\left(-\frac{(4|B_v^r|(2\Delta)^{-\ell})^2}{(14c+37)|B_v^r|(2\Delta)^{-\ell}}\right)
      \leq \exp\left(-\frac{16 |B_v^r|(2\Delta)^{-\ell}}{14c+37}\right)
      \leq \exp\left(-\log^4n\right);
   $$
   otherwise we bound above the exponential term by
   $$
      \exp\left(-\frac{(\frac{\log^7n}{\ell})^2}{(14c+37)\log^7n\log\log n}\right)
      \leq \exp\left(-\frac{\log^7n}{(14c+37)\ell^2\log\log n}\right)
      \leq \exp\left(-\log^4n\right),
   $$
   where the last inequality holds for all large enough $n$ since $\ell\leq \ell_1=O(\log n)$.
\end{proof}

\begin{proof}[{\bf Proof of Theorem~\ref{thm:expander}}]
   First note that the lower bound on the maximum load is established by Theorem~\ref{thm:lowerbounds}.
   So we now prove the upper bound.
   Let $v$ be an arbitrary vertex of $V$.
   We start the proof by showing that there exist positive constants $C$ and $c$ such that,
   for all large enough $n$, it holds that
   \begin{equation}
      \Pro{\sum_{u\in B_v^{r_0}} X_u^{(n)} > C |B_v^{r_0}|} \leq \exp\left(-c \log^3n\right).
      \label{eq:mainstep}
   \end{equation}
   Then, it follows by Lemma~\ref{lem:characterization} that
   $$
      X_v^{(n)} \leq C + \sum_{i=0}^{r_0} \frac{|N_v^i|}{|B_v^{r_0}|}.
   $$
   Note that $\sum_{i=0}^r \frac{|N_v^i|}{|B_v^{r_0}|}$ is the average distance between $v$ and a vertex in $B_v^{r_0}$, which is at most $r_0$. Therefore,
   we obtain $X_v^{(n)} \leq C + r_0$. By the definition of $r_0$, we have $|B_v^{r_0}| \leq \Delta \log^{10}n$.
   Combing this with the exponential growth property of $G$ yields
   $$
      r_0 \leq \frac{1}{\phi}   \log |B_v^{r_0}| = \Oh(\log\log n),
   $$
   where $\phi$ is  the parameter defined by
   \eq{expgrowth}.

   It remains to establish~\eqref{eq:mainstep}. First note that
   $$
      \sum_{u\in B_v^{r_0}} X_u^{(n)}
      = \sum_{\ell=1}^\infty |\Lambda_{r_0,\ell}|
      \leq |B_v^{r_0}|\ell_0 + \sum_{\ell=\ell_0+1}^\infty |\Lambda_{r_0,\ell}|.
   $$
   Therefore, we have that
   $$
      \Pro{\sum_{u\in B_v^{r_0}} X_u^{(n)} > C |B_v^{r_0}|}
      \leq \Pro{\sum_{\ell=\ell_0}^\infty |\Lambda_{r_0,\ell}| > (C-\ell_0)|B_v^{r_0}|}.
   $$
   Now, if $|\Lambda_{r_0,\ell}| \leq 8|B_v^{r_0}|(2\Delta)^{-\ell}+\frac{\log^7 n}{\ell}$ for all $\ell=\ell_0,\ell_0+1,\ldots,\ell_1$, then
   $\sum_{\ell=\ell_0}^\infty |\Lambda_{r_0,\ell}| \leq (C-\ell_0)|B_v^{r_0}|$ for some large enough $C$ since $|B_v^{r_0}|\geq \log^{10}n$.
   For any given $\ell$, using Lemma~\ref{lem:lambda}, we have
   \begin{equation}
      \Pro{|\Lambda_{r_0,\ell}| > 8|B_v^{r_0}|(2\Delta)^{-\ell}+\frac{\log^7 n}{\ell}}
      \leq \exp(-\log^4n) + n^2 \cdot \Pro{(L_{r_0,\ell-1}^\minuslast)^\mathrm{c}}.
      \label{eq:wantedbd}
   \end{equation}
   By definition of $L_{r,\ell}^\minuslast$ (cf.~\eqref{eq:defLr}), we have that
   \begin{align*}
      \Pro{(L_{r_0,\ell-1}^\minuslast)^\mathrm{c}}
      &\leq \Pro{|\Lambda_{r_0,\ell-1}^\minuslast| > 8|B_v^{r_0}|(2\Delta)^{-\ell+1}+\frac{\log^7 n}{\ell-1}}+\Pro{(L_{r_0,\ell-2}^\minuslast)^\mathrm{c}}\\
      &\leq \Pro{|\Lambda_{r_0,\ell-1}| > 8|B_v^{r_0}|(2\Delta)^{-\ell+1}+\frac{\log^7 n}{\ell-1}}+\Pro{(L_{r_0,\ell-2}^\minuslast)^\mathrm{c}}\\
      &\leq \exp(-\log^4n) + (n^2+1)\Pro{(L_{r_0,\ell-2}^\minuslast)^\mathrm{c}},
   \end{align*}
   where the second inequality follows since $\Lambda_{r_0,\ell-1}^\minuslast\subseteq \Lambda_{r_0,\ell-1}$ and last inequality follows from Lemma~\ref{lem:lambda}.
   Applying this into~\eqref{eq:wantedbd}, we have
   \begin{align*}
      &\Pro{|\Lambda_{r_0,\ell}| > 8|B_v^{r_0}|(2\Delta)^{-\ell}+\frac{\log^7 n}{\ell}}\\
      &\leq \exp(-\log^4n) + (n^2+1)\exp(-\log^4n) + (n^2+1)^2\Pro{(L_{r_0,\ell-2}^\minuslast)^\mathrm{c}}\\
      &\leq \sum_{k=0}^{\ell-\ell_0} (n^2+1)^k \exp(-\log^4n) + (n^2+1)^{\ell-\ell_0+1} \Pro{(L_{r_0+1,\ell_1}^\minuslast)^\mathrm{c}}.
   \end{align*}
   Using the same argument, we obtain for any $r$ that
   \begin{align*}
      \Pro{(L_{r,\ell_1}^\minuslast)^\mathrm{c}}
      &\leq \sum_{k=0}^{\ell_1-\ell_0} (n^2+1)^k \exp(-\log^4n) + (n^2+1)^{\ell_1-\ell_0+1} \Pro{(L_{r+1,\ell_1}^\minuslast)^\mathrm{c}}\\
      &\leq \sum_{j=0}^{R-r-1} \sum_{k=0}^{\ell_1-\ell_0} (n^2+1)^{\ell_1j+ k} \exp(-\log^4n) + (n^2+1)^{\ell_1 (R-r)} \Pro{(L_{R}^\minuslast)^\mathrm{c}}.
   \end{align*}
   Using \lemref{larger},
   \[
    \Pro{(L_{R}^\minuslast)^\mathrm{c}} \leq 2 n^{-\log^5 n},
   \]
   and
   plugging this into~\eqref{eq:wantedbd}, and using the union bound over $\ell$, we obtain that
   $$
      \Pro{ \bigcup\nolimits_{\ell=\ell_0}^{\ell_1}\Big\{|\Lambda_{r_0,\ell}| > 8|B_v^{r_0}|(2\Delta)^{-\ell}+\frac{\log^7 n}{\ell}\Big\}}\leq \exp(-c \log^3n)
   $$
   for
   some positive constant $c$, which establishes~(\ref{eq:mainstep}).
\end{proof}

\section{Grid graphs} \label{sec:grid}

In this section, we analyze the maximum load of the local allocation process on any $d$-dimensional grid, where $d$ is an arbitrary constant.
We show that the maximum load is $\Theta \Big( \left( \frac{\log n}{\log \log n} \right)^{\frac{1}{d+1}} \Big)$.
Interestingly, the analysis on the grid turns out to be much easier than the analysis on expander graphs,
as on grid graphs the number of paths the local search could follow is much smaller.

Formally, we define the $d$-dimensional grid by the vertex set
$V=\{u: u=(u_1,\ldots,u_d), u_i=0,\ldots, n^{1/d}-1\}$ and edge set $E=\{ \{u, v\} : \mathrm{dist}(u,v)=1 \}$, where
\[
\mathrm{dist}(u,v)=\sum_{i=1}^d\mathrm{dist}(u_i,v_i),\quad\mbox{and}\quad
\mathrm{dist}(u_i,v_i)= \min\left\{|u_i-v_i|, n^{1/d}-|u_i-v_i|\right\}.
\]

\begin{proof}[{\bf Proof of Theorem~\ref{thm:torus}}]
We start with the \emph{upper bound}.
Within this proof, we use the following notation:
\[
\widetilde{B}_u^r:=\{v\in V: \mathrm{dist}(u_i,v_i)\leq r, \forall i=1,\ldots,d\}.
\]
Roughly speaking, $\widetilde{B}_{u}^r$ can be seen as the ``$\ell_\infty$-version'' of the set $B_u^r$ used in \secref{expander}.
Note that for any $r \leq n^{1/d}/2-1/2$, $|\widetilde{B}_{u}^r| = (2r+1)^d$.
We first define an event that gives an upper bound for the number of balls born in $\widetilde{B}_{u}^r$ for various $u$ and $r$:
\begin{align*}
  \mathcal{E} &:= \bigcap_{u \in V} \bigcap_{r=(4d)^d \big(\frac{\log n}{\log \log n}\big)^{\frac{1}{d+1}}}^{n^{1/d}/2}
     \left\{  \sum_{v \in \widetilde{B}_{u}^r} Z_v^{(n)} \leq \rho(r) \right\},
\end{align*}
where $\rho(r) := 4\ce\cdot (d+1) (\frac{\log n}{\log \log n})^{\frac{1}{d+1}}  \cdot (3 r)^{d}$
and $Z_v^{(n)}=\sum_{i=1}^n \ind{ U_i=v  }$ is the number of balls born on $v$ during the first $n$ rounds.

To prove that $\mathcal{E}$ holds with high probability, fix any vertex $u \in V$ and $r\geq (4d)^d (\frac{\log n}{\log \log n})^{\frac{1}{d+1}}$. We have
\begin{align*}
 \lefteqn{\Pro{	 \sum_{v \in \widetilde{B}_{u}^r} Z_v^{(n)} \geq \rho(r)		}}\\
 &\leq \binom{n}{\rho(r)} \left( \frac{|\widetilde{B}_{u}^r|}{n} \right)^{\rho(r)}
    \leq \left( \frac{\ce\cdot n}{\rho(r)} \right)^{\rho(r)} \left( \frac{(2r+1)^{d}}{n} \right)^{\rho(r)} \\
 &\leq \left( \frac{\ce \cdot (3r)^{d}}{4\ce\cdot (d+1) (\frac{\log n}{\log \log n})^{\frac{1}{d+1}}  \cdot  (3r)^{d} }  \right)^{4 \ce \cdot (d+1) (\frac{\log n}{\log \log n})^{\frac{1}{d+1}} \cdot (3 r)^{d}} \leq n^{-3}.
\end{align*}
Taking the union bound over the $n$ vertices and at most $n/2$ possible values for $r$ yields $\Pro{ \mathcal{E} } \geq 1 - n^{-1}$.

Assuming that $\mathcal{E}$ occurs, we now infer the upper bound on the maximum load. Assume for the sake of contradiction that the maximum load is in the interval $[\alpha/2,
\alpha]$ where $\alpha$ is any value larger than $16\ce\cdot (16d)^d (d+1) ( \frac{\log n}{\log \log n})^{\frac{1}{d+1}}$. Let $u \in V$ be a vertex with $X_u^{(n)} \in [\alpha/2,\alpha]$. Since the maximum load is at most $\alpha$, only balls that are born in $\widetilde{B}^{2\alpha}_{u}$ can reach $\widetilde{B}^{\alpha}_{u}$. Since $\mathcal{E}$ occurs, we know for $r = 2 \alpha$ that
\begin{align}
    \sum_{v \in \widetilde{B}^{2 \alpha}_{u}} Z_v^{(n)} &\leq  4\ce\cdot (d+1) \left(\frac{\log n}{\log \log n}\right)^{1/(d+1)} \cdot (6 \alpha)^{d}. \label{eq:contraone}
\end{align}
On the other hand, if $u$ has load at least $\alpha/2$, then
\begin{align}
  \sum_{v \in \widetilde{B}^{2 \alpha}_{u}} Z_v^{(n)}
  &\geq \left|\widetilde{B}^{\alpha/(4d)}_{u} \right| \cdot \left( \frac{\alpha}{2} - d \cdot \frac{\alpha}{4d} \right) \notag \\
  &\geq \left( \frac{\alpha}{2d} \right)^{d} \cdot \frac{\alpha}{4} = \frac{1}{4} \cdot \frac{\alpha}{(16d)^d} \cdot    (8 \alpha)^{d} \notag \\
  &\geq 4 \ce\cdot(d+1)\cdot \left(\frac{\log n}{\log \log n} \right)^{\frac{1}{d+1}} \cdot (8 \alpha)^{d},
  \label{eq:contratwo}
\end{align}
where the last step used our lower bound on $\alpha$. The desired contradiction follows now from (\ref{eq:contraone}) and (\ref{eq:contratwo}), and the proof of the upper bound is complete.
%

Now we proceed to establish the \emph{lower bound}.
It is a well-known fact (cf.~\cite[Lemma~5.12]{MU05})
that with probability at least $1-n^{-1}$,
there is a vertex $u \in V$ on which at least $\frac{\log n}{\log \log n}$ balls are born.
Applying \lemref{characterizationlower} with $S=\{u \}$, $\Phi_{S} = \frac{\log n}{\log \log n} $ implies that the maximum load $\beta:=X_{\max}^{(n)}$ satisfies
\[
  \beta \cdot |B_{u}^{\beta}| \geq  \frac{\log n}{\log \log n}.
  \]
  Hence, as $|B_{u}^{\beta}| \leq |\widetilde{B}_{u}^{\beta}| \leq (2 \beta + 1)^{d}$ and $d$ is a constant, we obtain that $\beta=\Omega\left( \Big(\frac{\log n}{\log \log n}\Big)^{\frac{1}{d+1}}\right)$.
\end{proof}


\section{Dense graphs}\label{sec:densegraph}

In this section, we analyze {\em dense} graphs which we define as graphs where
the minimum degree is $\Omega(\log n)$ and the ratio between the maximum and minimum degrees is constant.
This includes, for instance, the $\log n$-dimensional hypercube and Erd\H{o}s-R\'enyi random graphs with average degree $(1+\epsilon) \log n$, for any $\epsilon>0$.
The key idea of the analysis is that as long as less than $\alpha$ balls are allocated, where $\alpha < n$, every vertex has a constant fraction of neighbors which have received no ball, and hence, the maximum load is bounded by $1$.
Lemma~\ref{lem:subadditivity}
implies then that after $n$ balls are allocated, the maximum load is at most $n/\alpha$.
To make the analysis work, we need to assume that ties are broken uniformly at random; i.e.,
whenever a ball has more than one vertex to be forwarded to, the vertex is chosen independently and
uniformly at random among the set of possible vertices.

\begin{proof}[{\bf Proof of Theorem~\ref{thm:dense}}]
We divide the process of allocating the $n$ balls into different phases, where each phase allocates a batch of consecutive $\alpha \leq n$ balls
(hence the number of phases is $\lceil n/ \alpha \rceil$).
Then, by subadditivity (cf.\ Lemma~\ref{lem:subadditivity}), the maximum load at the end is at most $\lceil n/ \alpha \rceil$ times
the maximum load of a single phase.

Let $G$ be an almost regular graph with minimum degree $\delta = c \cdot \log n$ and maximum degree $\Delta \leq C \cdot \delta$,
where $c > 0$ is any value bounded below by a constant and $C \geq 1$ is a constant. For any load assignment of the vertices $(x_u)_{u \in V}$, we define the exponential potential as:
\[
 \Phi((x_u)_{u \in V}) := \sum_{u \in V} \exp\left(\sigma \cdot \sum_{v \in N_u} x_v\right),
\]
where $\sigma := \max \{ 4 \log (n) / \delta, 1 \} = \Oh(1)$.
We also define for any $1 \leq t \leq n$,
\[
 \Phi^{(t)} := \Phi((X_u^{(t)})_{u \in V}) = \sum_{u \in V} \exp\left(\sigma \cdot \sum_{v \in N_u} X_v^{(t)}\right),
\]
hence $\Phi^{(0)} = n$.
Our goal is to bound the expected multiplicative increase in $\Phi^{(t+1)}$ compared to $\Phi^{(t)}$. In order to do that, we will actually also assume that $\Phi^{(t)}$ is small.

Specifically, assume that $(x_u)_{u \in V}$ be any vector in $(\mathbb{N} \cup \{0\})^{n}$ such that $\Phi((x_u)_{u \in V}) \leq n \cdot \ce^{\delta/2}$ and suppose that the load vector at the end of round $t$ is $(x_u)_{u \in V}$, i.e., $X^{(t)} = x$.
Then this implies for every vertex $u \in V$,
\[
  \exp\left( \sigma \cdot \sum_{v \in N_u} X_v^{(t)} \right) \leq n \cdot \ce^{\delta /2} = \ce^{\log n + \delta/2},
\]
and consequently, $\sum_{v \in N_u} X_v^{(t)} \leq (1/\sigma) \cdot (\log n + \delta/2) \leq \delta/4 + \delta/2 = (3/4) \delta$.
Hence, there are at least $\deg(u) - (3/4) \delta \geq (1/4) \delta$ neighbors of $u$ which have no ball, where $\deg(u)$ is the degree of vertex $u$. In particular, this implies that the next ball $t+1$ will be allocated either on its birthplace or at a direct neighbor.
Therefore,
\begin{align}
 \lefteqn{\Pro{ \sum_{v\in N_u} X_v^{(t+1)}  =  \sum_{v\in N_u} X_v^{(t+1)} + 1 \, \mid \, X^{(t)} = x  } }\notag\\
 &\leq \sum_{v \in N_u} \sum_{w \in N_v\cup \{v\}} \Pro{ \mbox{ball $t+1$ born at $w$ and allocated on $v$} \, \mid \, X^{(t)} = x } \notag \\
 &\leq \sum_{v \in N_u} \left( \sum_{w \in N_v} \left( \frac{1}{n} \cdot \frac{4}{\delta} \right) + \frac{1}{n} \right) \leq \frac{5 C \Delta}{n}. \label{eq:boundprob}
\end{align}
This yields,\begin{align}
\lefteqn{ \Ex{   \Phi^{(t+1)} \, \mid \,  X^{(t)}=x   }} \notag \\ &\leq  \sum_{u \in V } \Biggl( \Pro{  \sum_{v\in N_u} X_v^{(t+1)} =  \sum_{v\in N_u} X_v^{(t)} + 1 \, \mid \, X^{(t)}=x  }  \cdot \exp\left(\sigma \cdot \big(\sum_{v\in N_u} X_v^{(t)}+1\big)\right) \notag \\ &\quad\qquad\, + \Pro{ \sum_{v\in N_u} X_v^{(t+1)}  =  \sum_{v\in N_u} X_v^{(t)} \, \mid \, X^{(t)}=x } \cdot \exp\left(\sigma \cdot  \sum_{v\in N_u} X_v^{(t)}\right) \Biggr) \notag \\
 &=  \sum_{u \in V } \left( \ce^{\sigma} \cdot \Pro{  \sum_{v\in N_u} X_v^{(t+1)} = \sum_{v\in N_u} X_v^{(t)} + 1 \, \mid \, X^{(t)}=x  } + 1 \right) \cdot \exp\left(\sigma \cdot \sum_{v\in N_u} X_v^{(t)}\right) \notag \\
 &\leq  \left(1 + \ce^{\sigma} \cdot \frac{5 C \Delta}{n} \right) \cdot \Phi^{(t)}, \label{eq:potential}
\end{align}
where (\ref{eq:potential}) follows from (\ref{eq:boundprob}).
Note that if we only consider the allocation of $\alpha := n/(\ce^{\sigma} \cdot 25 C^2  ) = \Theta(n)$ balls, then we have
\[
\left(1 +  \ce^{\sigma} \cdot \frac{5 C \Delta}{n}  \right)^{\alpha} \leq \ce^{\Delta / (5 C) } \leq \ce^{\delta / 4 }.
\]
Define $\Psi^{(t)} := \min \{  \Phi^{(t)}, n \cdot \ce^{\delta/2} \}$. Then, since $\Phi^{(t)}$ is increasing in $t$, \eq{potential} yields
\begin{align*}
 \Ex{ \Psi^{(t+1)}  } &\leq \left(1 + \ce^{\sigma} \cdot \frac{5 C \Delta}{n} \right) \cdot \Psi^{(t)}
\end{align*}
and thus inductively,
$
 \Ex{ \Psi^{(\alpha)}  } \leq \left(1 + \ce^{\sigma} \cdot \frac{5 C \Delta}{n} \right)^{\alpha} \cdot \Psi^{(0)} \leq \ce^{\delta/4} \cdot n.
$
Hence, applying Markov's inequality gives
$
 \Pro{  \Psi^{(\alpha)} < \ce^{\delta/2} \cdot n        } \geq 1 - \ce^{-\delta/4}.
$
By definition of $\Psi^{(\alpha)}$,
if $\Psi^{(\alpha)} < \ce^{\delta/2} \cdot n$, then $\Psi^{(\alpha)} = \Phi^{(\alpha)}$. Hence,
$
 \Pro{  \Phi^{(\alpha)} <  \ce^{\delta/2} \cdot n        } \geq 1 - \ce^{-\delta/4},
$
as required. If $\Phi^{(\alpha)} <  \ce^{\delta/2} \cdot n$ occurs, then since every
vertex has at least one neighbor with load zero, the maximum load after the allocation of $\alpha$ balls is $1$.
Then, we use subadditivity (cf.\ Lemma~\ref{lem:subadditivity}) to conclude that the maximum load after all $\lceil n/\alpha \rceil \cdot \alpha$ balls are allocated
is at most $1 \cdot \lceil n/\alpha \rceil$
with probability at least $1- \lceil n/\alpha \rceil \cdot o(1) = 1-o(1)$.
\end{proof}

\section{Impact of tie-breaking rules}\label{sec:tiebreaking}

\begin{proof}[{\bf Proof of Theorem~\ref{thm:tiebreaking}}]
We now analyze the effect of employing different tie-breaking rules.
We first describe the construction of the graph $G$, which will be a $d$-regular graph, where $d=\omega(1)$ as $n \rightarrow \infty$. Additionally, we may assume that $d \leq \sqrt{n}$, since otherwise the claimed lower bound is trivial.
For the construction of $G$, we assume that there is an integer $k$ such that $1 + \sum_{i=0}^{k-1} d \cdot (d-1)^{i} = n$. Note that $k=\Theta( \log n / \log d)$. Then, let $G$ be a balanced tree with root $s$ so that
\[
 |N_s^{i}| = d \cdot (d-1)^{i-1} \quad \, \mbox{for any $i \geq 1$}.
\]
Hence $G$ is a tree where all vertices except for the leaves and the root have $d-1$ successors; the root has $d$ successors, and the leaves have no successor.
Hence, the root has degree $d$, the inner vertices have degree $d$ as well and the leaves have degree $1$.
Further, note that the number of leaves is $d \cdot (d-1)^{k-1}$.
In order to make the graph $d$-regular, we simply add edges among the leaves in $G$ so that, after all edges have been added,
every leaf has degree $d$ (this is possible, since the number of edges to add is smaller than the total number of leaves).



We choose $\alpha := \min \left\{  (d-1)^{1/4}, k - 2 \right\} $.
The process of allocating the $n$ balls will be divided into $\alpha$ phases and in each phase we consider the allocation of $n/\alpha$ balls. To prove the desired lower bound, we focus on the vertices in $B_{s}^{2\alpha}$.

Next, we define  an event that essentially shows that there are always enough balls so that the tie-breaking rule can send balls towards the root: 
\begin{align*}
 \mathcal{E} := \bigcap_{p=1}^{\alpha} \bigcap_{\ell=1}^{2 \alpha} \left\{ \forall u \in N^{\ell}_s \, \exists	 v \in N_u	\cap N_s^{\ell+1} \colon Z_{p}(v) \geq 2 \right\},
\end{align*}
where $Z_{p}(v) := \sum_{t=(p-1)\cdot (n/\alpha)+1}^{p \cdot (n/\alpha)} \ind{ U_t = v }$ is the number of balls born on $v$ in phase $p$. Hence the event $\mathcal{E}$ means that for each vertex in $N^{\ell}_s$, there is in each phase at least one neighbor in $N^{\ell+1}_s$ on which two balls are born; hence, for at least one ball we may be able to use the tie-breaking rule and forward the ball towards the root $s$.

Let us estimate the probability that the event $\mathcal{E}$ occurs. First, for any fixed $u \in N_s^{\ell}$ and $v \in N_u	\cap N^{\ell+1}_s$
\begin{align*}
 \Pro{	Z_{p}(v) \geq 2		} &\geq \binom{n/\alpha}{2}\cdot \frac{1}{n^2} \cdot \left(1 - \frac{1}{n} \right)^{n/\alpha-2} \geq \frac{1}{8 \alpha^2}.
\end{align*}
Since the events $\{ Z_{p}(v) \geq 2 \}_{v \in N_u}$,
are negatively correlated, we have that, for any fixed $u \in N^{\ell}_s$,
\begin{align*}
 \Pro{	\exists v \in N_u \cap N^{\ell+1}_s: Z_{p}(v) \geq 2		} &\geq 1 - \left(1 - \frac{1}{8 \alpha^2} \right)^{d-1} \geq 1 - \exp\left( - \frac{d-1}{8 \alpha^2} \right).
\end{align*}
Hence,
\begin{align}
 \Pro{ \mathcal{E} } &\geq 1 - \sum_{p=1}^{\alpha} \sum_{\ell=1}^{2\alpha} \sum_{u \in N^{\ell}_s} \Pro{\neg	\left( \exists v \in N_u \cap N^{\ell+1}_s: Z_{p}(v) \geq 2		\right) } \notag \\
 &\geq 1 - \alpha \cdot \sum_{\ell=1}^{2 \alpha} |N_{s}^{\ell}| \cdot \exp\left( - \frac{d-1}{8 \alpha^2} \right)   \geq 1 - \exp\left( - \frac{d-1}{8 \alpha^2}	 \right) \cdot \alpha \cdot 2 d (d-1)^{2\alpha-1}. \label{eq:proba}
\end{align}

We now claim that the last term in \eq{proba} is $1-o(1)$. To this end, recall the choice of $\alpha$ and $d = \omega(1)$.
First, since $\alpha \leq (d-1)^{1/4}$, $\exp\left( - \frac{d-1}{8 \alpha^2}	 \right) \leq \exp \big( - \frac{(d-1)^{1/2}}{8}  \big)$,
whereas $\alpha \cdot 2d \cdot(d-1)^{2 \alpha-1} \leq 2 d^{3 \alpha} \leq 2 d^{ 3 (d-1)^{1/4}} =2 \cdot \exp \big( \log d \cdot 3 (d-1)^{1/4} \big)$.
Hence, as $d \rightarrow \infty$, the probability on the right-hand side in \eq{proba} is $1 - o(1)$; i.e., we  have shown that
\begin{align*}
 \Pro{ \mathcal{E} } &\geq 1-o(1).
 \end{align*}

  \newcommand{\bin}[1]{         \draw [color=black, fill=white, opacity=1] (#1)++(0.196,1.4) to ++(0,-1.2)  to ++(-0.40,0) to ++ (0,1.2);}

  \newcommand{\firstball}[1]{ \shade[ball color=yellow, opacity=0.5] (#1)++(0,0.42) circle (.17cm);
  }
  \newcommand{\firstrball}[1]{ \shade[ball color=red, opacity=0.5] (#1)++(0,0.42) circle (.17cm);
  }
  \newcommand{\secondball}[1]{
  \shade[ball color=yellow, opacity=0.5] (#1)++(0,0.78) circle (.17cm);
  }
   \newcommand{\secondrball}[1]{
  \shade[ball color=red, opacity=0.5] (#1)++(0,0.78) circle (.17cm);
  }
   \newcommand{\thirdball}[1]{
  \shade[ball color=yellow, opacity=0.5] (#1)++(0,1.14) circle (.17cm);
  }
     \newcommand{\thirdrball}[1]{
  \shade[ball color=red, opacity=0.5] (#1)++(0,1.14) circle (.17cm);
  }
  \newcommand{\gthirdball}[1]{
   \draw[ball color=yellow, dotted, fill=white,opacity=0.5] (#1)++(0,1.14) circle (.17cm);
  }

   \newcommand{\fourthball}[1]{
 \shade[ball color=yellow, opacity=0.5] (#1)++(0,1.8) circle (.17cm);
  }

  \begin{figure}[h]
  \begin{tikzpicture}[auto, xscale=0.55, yscale=0.55,  knoten/.style={
           draw=black, fill=black, thin, circle, inner sep=0.05cm}, inv/.style={draw=white, fill=white}]
           \draw[white] (12,6) -- (17,6);
           \node[knoten] (1) at (8,14) [label=left:$s$]{};
           \node[knoten] (2) at (4,12) [label=left:$$]{};
           \node[knoten] (3) at (8,12) [label=left:$$]{};
           \node[knoten] (4) at (12,12) [label=left:$$]{};
           \node[knoten] (5) at (3,10) [label=left:$$]{};
           \node[knoten] (6) at (5,10) [label=left:$$]{};
           \node[knoten] (7) at (7,10) [label=left:$$]{};
           \node[knoten] (8) at (9,10) [label=left:$$]{};
           \node[knoten] (9) at (11,10) [label=left:$$]{};
           \node[knoten] (10) at (13,10) [label=left:$$]{};
           \node[knoten] (11) at (2.5,8) [label=left:$$]{};
           \node[knoten] (12) at (3.5,8) [label=left:$$]{};
           \node[knoten] (13) at (4.5,8) [label=left:$$]{};
           \node[knoten] (14) at (5.5,8) [label=left:$$]{};
           \node[knoten] (15) at (6.5,8) [label=left:$$]{};
           \node[knoten] (16) at (7.5,8) [label=left:$$]{};
           \node[knoten] (17) at (8.5,8) [label=left:$$]{};
           \node[knoten] (18) at (9.5,8) [label=left:$$]{};
           \node[knoten] (19) at (10.5,8) [label=left:$$]{};
           \node[knoten] (20) at (11.5,8) [label=left:$$]{};
           \node[knoten] (21) at (12.5,8) [label=left:$$]{};
           \node[knoten] (22) at (13.5,8) [label=left:$$]{};
           \node[knoten] (23) at (2.25,6) []{};
           \node[knoten] (24) at (2.75,6) []{};
           \node[knoten] (25) at (3.25,6) []{};
           \node[knoten] (26) at (3.75,6) []{};
           \node[knoten] (27) at (4.25,6) []{};
           \node[knoten] (28) at (4.75,6) []{};
           \node[knoten] (29) at (5.25,6) []{};
           \node[knoten] (30) at (5.75,6) []{};
           \node[knoten] (31) at (6.25,6) []{};
           \node[knoten] (32) at (6.75,6) []{};
           \node[knoten] (33) at (7.25,6) []{};
           \node[knoten] (34) at (7.75,6) []{};
           \node[knoten] (35) at (8.25,6) []{};
           \node[knoten] (36) at (8.75,6) []{};
           \node[knoten] (37) at (9.25,6) []{};
           \node[knoten] (38) at (9.75,6) []{};
           \node[knoten] (39) at (10.25,6) []{};
           \node[knoten] (40) at (10.75,6) []{};
           \node[knoten] (41) at (11.25,6) []{};
           \node[knoten] (42) at (11.75,6) []{};
           \node[knoten] (43) at (12.25,6) []{};
           \node[knoten] (44) at (12.75,6) []{};
           \node[knoten] (45) at (13.25,6) []{};
           \node[knoten] (46) at (13.75,6) []{};
           \draw (1) -- (2);
           \draw (1) -- (3);
           \draw (1) -- (4);
           \draw (2) -- (5);
           \draw (2) -- (6);
           \draw (3) -- (7);
           \draw (3) -- (8);
           \draw (4) -- (9);
           \draw (4) -- (10);
           \draw (5) -- (11);
           \draw (5) -- (12);
           \draw (6) -- (13);
           \draw (6) -- (14);
           \draw (7) -- (15);
           \draw (7) -- (16);
           \draw (8) -- (17);
           \draw (8) -- (18);
           \draw (9) -- (19);
           \draw (9) -- (20);
           \draw (10) -- (21);
           \draw (10) -- (22);
           \draw (11) -- (23);
           \draw (11) -- (24);
           \draw (12) -- (25);
           \draw (12) -- (26);
           \draw (13) -- (27);
           \draw (13) -- (28);
           \draw (14) -- (29);
           \draw (14) -- (30);
           \draw (15) -- (31);
           \draw (15) -- (32);
           \draw (16) -- (33);
           \draw (16) -- (34);
           \draw (17) -- (35);
           \draw (17) -- (36);
           \draw (18) -- (37);
           \draw (18) -- (38);
           \draw (19) -- (39);
           \draw (19) -- (40);
           \draw (20) -- (41);
           \draw (20) -- (42);
           \draw (21) -- (43);
           \draw (21) -- (44);
           \draw (22) -- (45);
           \draw (22) -- (46);
           \foreach \x in {1,2,...,46}
            {\bin{\x};
            }
            \foreach \x in {1,...,46}
            {
			 \firstrball{\x};   }
\end{tikzpicture}
  \begin{tikzpicture}[auto, xscale=0.55, yscale=0.55,  knoten/.style={
           draw=black, fill=black, thin, circle, inner sep=0.05cm}, inv/.style={draw=white, fill=white}]
           \draw[white] (12,6) -- (17,6);
           \node[knoten] (1) at (8,14) [label=left:$s$]{};
           \node[knoten] (2) at (4,12) [label=left:$$]{};
           \node[knoten] (3) at (8,12) [label=left:$$]{};
           \node[knoten] (4) at (12,12) [label=left:$$]{};
           \node[knoten] (5) at (3,10) [label=left:$$]{};
           \node[knoten] (6) at (5,10) [label=left:$$]{};
           \node[knoten] (7) at (7,10) [label=left:$$]{};
           \node[knoten] (8) at (9,10) [label=left:$$]{};
           \node[knoten] (9) at (11,10) [label=left:$$]{};
           \node[knoten] (10) at (13,10) [label=left:$$]{};
           \node[knoten] (11) at (2.5,8) [label=left:$$]{};
           \node[knoten] (12) at (3.5,8) [label=left:$$]{};
           \node[knoten] (13) at (4.5,8) [label=left:$$]{};
           \node[knoten] (14) at (5.5,8) [label=left:$$]{};
           \node[knoten] (15) at (6.5,8) [label=left:$$]{};
           \node[knoten] (16) at (7.5,8) [label=left:$$]{};
           \node[knoten] (17) at (8.5,8) [label=left:$$]{};
           \node[knoten] (18) at (9.5,8) [label=left:$$]{};
           \node[knoten] (19) at (10.5,8) [label=left:$$]{};
           \node[knoten] (20) at (11.5,8) [label=left:$$]{};
           \node[knoten] (21) at (12.5,8) [label=left:$$]{};
           \node[knoten] (22) at (13.5,8) [label=left:$$]{};
           \node[knoten] (23) at (2.25,6) []{};
           \node[knoten] (24) at (2.75,6) []{};
           \node[knoten] (25) at (3.25,6) []{};
           \node[knoten] (26) at (3.75,6) []{};
           \node[knoten] (27) at (4.25,6) []{};
           \node[knoten] (28) at (4.75,6) []{};
           \node[knoten] (29) at (5.25,6) []{};
           \node[knoten] (30) at (5.75,6) []{};
           \node[knoten] (31) at (6.25,6) []{};
           \node[knoten] (32) at (6.75,6) []{};
           \node[knoten] (33) at (7.25,6) []{};
           \node[knoten] (34) at (7.75,6) []{};
           \node[knoten] (35) at (8.25,6) []{};
           \node[knoten] (36) at (8.75,6) []{};
           \node[knoten] (37) at (9.25,6) []{};
           \node[knoten] (38) at (9.75,6) []{};
           \node[knoten] (39) at (10.25,6) []{};
           \node[knoten] (40) at (10.75,6) []{};
           \node[knoten] (41) at (11.25,6) []{};
           \node[knoten] (42) at (11.75,6) []{};
           \node[knoten] (43) at (12.25,6) []{};
           \node[knoten] (44) at (12.75,6) []{};
           \node[knoten] (45) at (13.25,6) []{};
           \node[knoten] (46) at (13.75,6) []{};
           \draw (1) -- (2);
           \draw (1) -- (3);
           \draw (1) -- (4);
           \draw (2) -- (5);
           \draw (2) -- (6);
           \draw (3) -- (7);
           \draw (3) -- (8);
           \draw (4) -- (9);
           \draw (4) -- (10);
           \draw (5) -- (11);
           \draw (5) -- (12);
           \draw (6) -- (13);
           \draw (6) -- (14);
           \draw (7) -- (15);
           \draw (7) -- (16);
           \draw (8) -- (17);
           \draw (8) -- (18);
           \draw (9) -- (19);
           \draw (9) -- (20);
           \draw (10) -- (21);
           \draw (10) -- (22);
           \draw (11) -- (23);
           \draw (11) -- (24);
           \draw (12) -- (25);
           \draw (12) -- (26);
           \draw (13) -- (27);
           \draw (13) -- (28);
           \draw (14) -- (29);
           \draw (14) -- (30);
           \draw (15) -- (31);
           \draw (15) -- (32);
           \draw (16) -- (33);
           \draw (16) -- (34);
           \draw (17) -- (35);
           \draw (17) -- (36);
           \draw (18) -- (37);
           \draw (18) -- (38);
           \draw (19) -- (39);
           \draw (19) -- (40);
           \draw (20) -- (41);
           \draw (20) -- (42);
           \draw (21) -- (43);
           \draw (21) -- (44);
           \draw (22) -- (45);
           \draw (22) -- (46);
           \foreach \x in {1,2,...,46}
            {\bin{\x};
            }
            \foreach \x in {1,...,46}
            {
			 \firstball{\x};   }
			 			       \foreach \x in {1,...,10}
            {
			 \secondrball{\x};   }
\end{tikzpicture}
\begin{center}
 \begin{tikzpicture}[auto, xscale=0.55, yscale=0.55,  knoten/.style={
           draw=black, fill=black, thin, circle, inner sep=0.05cm}, inv/.style={draw=white, fill=white}]
           \draw[white] (12,16) -- (17,16);
           \node[knoten] (1) at (8,14) [label=left:$s$]{};
           \node[knoten] (2) at (4,12) [label=left:$$]{};
           \node[knoten] (3) at (8,12) [label=left:$$]{};
           \node[knoten] (4) at (12,12) [label=left:$$]{};
           \node[knoten] (5) at (3,10) [label=left:$$]{};
           \node[knoten] (6) at (5,10) [label=left:$$]{};
           \node[knoten] (7) at (7,10) [label=left:$$]{};
           \node[knoten] (8) at (9,10) [label=left:$$]{};
           \node[knoten] (9) at (11,10) [label=left:$$]{};
           \node[knoten] (10) at (13,10) [label=left:$$]{};
           \node[knoten] (11) at (2.5,8) [label=left:$$]{};
           \node[knoten] (12) at (3.5,8) [label=left:$$]{};
           \node[knoten] (13) at (4.5,8) [label=left:$$]{};
           \node[knoten] (14) at (5.5,8) [label=left:$$]{};
           \node[knoten] (15) at (6.5,8) [label=left:$$]{};
           \node[knoten] (16) at (7.5,8) [label=left:$$]{};
           \node[knoten] (17) at (8.5,8) [label=left:$$]{};
           \node[knoten] (18) at (9.5,8) [label=left:$$]{};
           \node[knoten] (19) at (10.5,8) [label=left:$$]{};
           \node[knoten] (20) at (11.5,8) [label=left:$$]{};
           \node[knoten] (21) at (12.5,8) [label=left:$$]{};
           \node[knoten] (22) at (13.5,8) [label=left:$$]{};
           \node[knoten] (23) at (2.25,6) []{};
           \node[knoten] (24) at (2.75,6) []{};
           \node[knoten] (25) at (3.25,6) []{};
           \node[knoten] (26) at (3.75,6) []{};
           \node[knoten] (27) at (4.25,6) []{};
           \node[knoten] (28) at (4.75,6) []{};
           \node[knoten] (29) at (5.25,6) []{};
           \node[knoten] (30) at (5.75,6) []{};
           \node[knoten] (31) at (6.25,6) []{};
           \node[knoten] (32) at (6.75,6) []{};
           \node[knoten] (33) at (7.25,6) []{};
           \node[knoten] (34) at (7.75,6) []{};
           \node[knoten] (35) at (8.25,6) []{};
           \node[knoten] (36) at (8.75,6) []{};
           \node[knoten] (37) at (9.25,6) []{};
           \node[knoten] (38) at (9.75,6) []{};
           \node[knoten] (39) at (10.25,6) []{};
           \node[knoten] (40) at (10.75,6) []{};
           \node[knoten] (41) at (11.25,6) []{};
           \node[knoten] (42) at (11.75,6) []{};
           \node[knoten] (43) at (12.25,6) []{};
           \node[knoten] (44) at (12.75,6) []{};
           \node[knoten] (45) at (13.25,6) []{};
           \node[knoten] (46) at (13.75,6) []{};
           \draw (1) -- (2);
           \draw (1) -- (3);
           \draw (1) -- (4);
           \draw (2) -- (5);
           \draw (2) -- (6);
           \draw (3) -- (7);
           \draw (3) -- (8);
           \draw (4) -- (9);
           \draw (4) -- (10);
           \draw (5) -- (11);
           \draw (5) -- (12);
           \draw (6) -- (13);
           \draw (6) -- (14);
           \draw (7) -- (15);
           \draw (7) -- (16);
           \draw (8) -- (17);
           \draw (8) -- (18);
           \draw (9) -- (19);
           \draw (9) -- (20);
           \draw (10) -- (21);
           \draw (10) -- (22);
           \draw (11) -- (23);
           \draw (11) -- (24);
           \draw (12) -- (25);
           \draw (12) -- (26);
           \draw (13) -- (27);
           \draw (13) -- (28);
           \draw (14) -- (29);
           \draw (14) -- (30);
           \draw (15) -- (31);
           \draw (15) -- (32);
           \draw (16) -- (33);
           \draw (16) -- (34);
           \draw (17) -- (35);
           \draw (17) -- (36);
           \draw (18) -- (37);
           \draw (18) -- (38);
           \draw (19) -- (39);
           \draw (19) -- (40);
           \draw (20) -- (41);
           \draw (20) -- (42);
           \draw (21) -- (43);
           \draw (21) -- (44);
           \draw (22) -- (45);
           \draw (22) -- (46);
           \foreach \x in {1,2,...,46}
            {\bin{\x};
            }
            \foreach \x in {1,...,46}
            {
			 \firstball{\x};
			 }
			 			       \foreach \x in {1,...,10}
            {
			 \secondball{\x};   }
			 			 			       \foreach \x in {1}
            {
			 \thirdrball{\x};   }

\end{tikzpicture}
\end{center}

\caption{Illustration of the $\alpha=3$ phases after which the root vertex $s$ has a load of at least $3$. The red color indicates the ball which are placed on the vertex in the recent phase. The existence of a red ball on a vertex, say, $u$, follows, since there exists at least one successor of $u$, say, $v \in N_u$, on which at least two balls are born in that phase.
}
\label{fig:tie}
\end{figure}
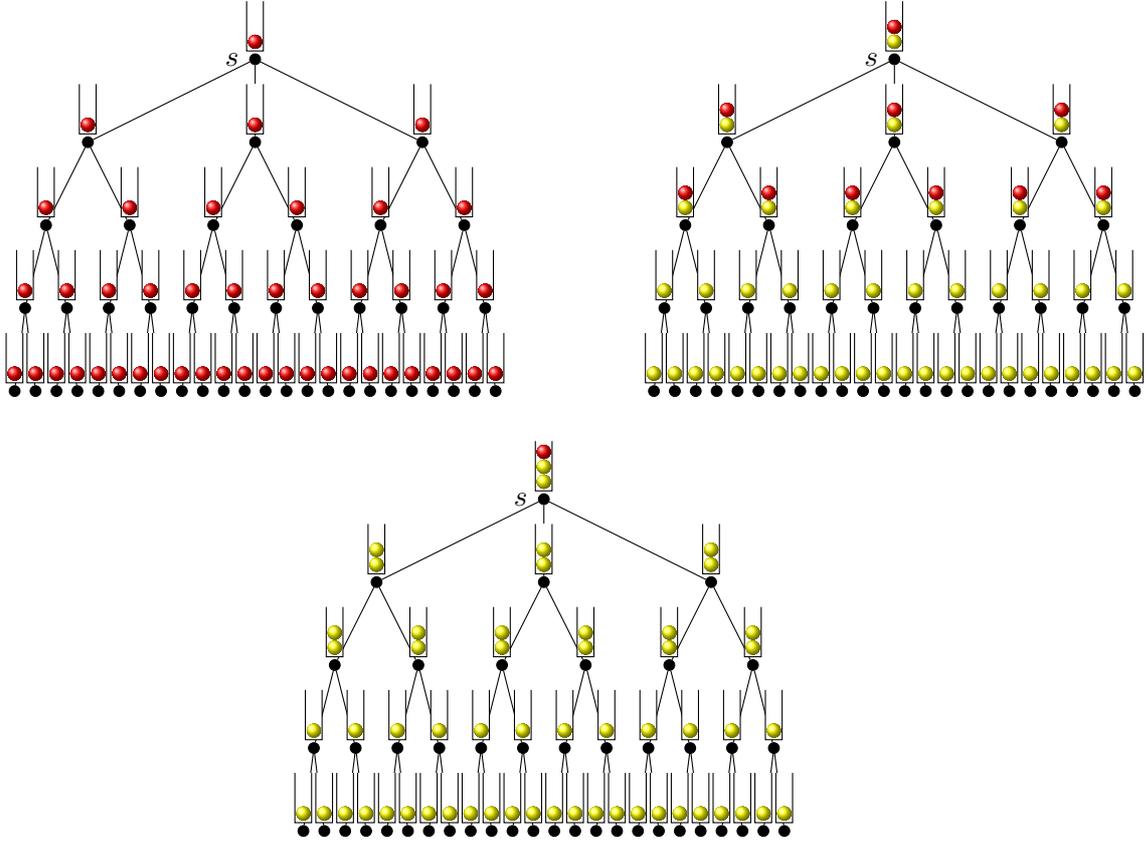

It remains to show that if the event $\mathcal{E}$ occurs, then it is possible to break the ties so that we end up with a maximum load of $\alpha$ after the
allocation of $n$ balls (see \figref{tie} for an illustration).
Our tie-breaking rule follows the simple strategy that, whenever possible, a ball is forwarded to a vertex closer to the root $s$.
By \lemref{monotonicity}, we may assume that balls are only generated in the set $B^{\alpha}_s$ and, in each phase,
every vertex in this set generates at most two balls. Then after the first phase is completed, every vertex in $B^{2 \alpha}_{s}$ contains at least one ball, since for each such vertex there is at least one neighbor in $B^{2 \alpha+1}_{s}$ on which two balls are born.
Moreover, in the second phase, every vertex $u$ in $B^{2\alpha-2}_{s}$ has at least one neighbor $v \in N_u$ in the next level such that:
(i) two balls are generated on $v$ in the second phase and (ii) all neighbors of $v$ and $v$ itself have at least one ball at the beginning of phase two.
Hence if ties are always broken in the direction towards the root,
every vertex in $B^{2 \alpha-2}_s$ will have at least two balls at the end of phase two.
Completing the induction, we conclude that at the end of phase $\alpha$, the vertex $s$ will have at least $\alpha$ balls.
Overall, we conclude that the maximum load is at least $\alpha$ whenever the event $\mathcal{E}$ occurs.
Since $\mathcal{E}$ occurs with probability $1-o(1)$, the proof is complete.
\end{proof}


\section{Lower bounds for sparse graphs}\label{sec:lowerbounds}

%
\begin{proof}[{\bf Proof of Theorem~\ref{thm:lowerbounds}}]

In the first part, we show that the maximum load after $n$ balls are allocated is $\Omega\big(\frac{\log\log n}{\log \Delta}\big)$, where $G$ is any graph with maximum degree $\Delta$. Our arguments are almost the same as in the proof of the lower bound of \thmref{torus}.
We are using again the fact that with probability at least $1-n^{-1}$,
there is a vertex $u \in V$ on which at least $\frac{\log n}{\log \log n}$ balls are born (cf.~\cite[Lemma~5.12]{MU05}). Then, applying \lemref{characterizationlower} with $S=\{u\}$, we obtain that the maximal load $\beta:=X_{\max}^{(n)}$ satisfies
\[
  \beta \cdot |B_{u}^{\beta}| \geq  \frac{\log n}{\log \log n}.
\]
Since $|B_{u}^{\beta}| \leq \Delta^{\beta}$, the above inequality implies that
\[
  \beta \cdot \Delta^{\beta} \geq \frac{\log n}{\log \log n},
\]
which in turn implies that $\beta=\Omega\left( \frac{\log \Big( \frac{ \log n}{\log \log n} \Big)}{\log \Delta} - \log(\beta) \right)$, i.e.,
$\beta = \Omega\left(  \frac{\log \log n}{\log \Delta}   \right)$.

Now, for the second part, we show that there exists a $d$-regular graph for which the maximum load after $n$ balls are allocated is
$\Omega \left(\sqrt{ \frac{\log n}{d \cdot \log \left(  \frac{\log n}{d} \right)} }\right)$.
We first describe the construction of the $d$-regular graph $G$.
First take $n/(d-1)$ disjoint cliques of size $d-1$ and arrange them in a cycle.
Then connect two cliques which are next to each other in the cycle by $d-1$ vertex-disjoint edges. This way we obtain a $d$-regular graph, which can be also defined as the Cartesian product of a cycle of length $n/(d-1)$ and a clique of size $d-1$.

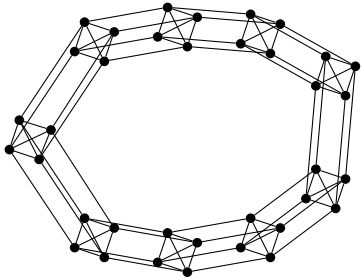
\begin{figure}[h]

\centering
\begin{tikzpicture}[auto, scale=1.3, rotate=90, bend right, bend angle=45, knoten/.style={
           draw=black, fill=black, thin, circle, inner sep=0.04cm}]
 \node[knoten] (51) at (5.95,5.05) {};
 \node[knoten] (52) at (6.05,5.35) {};
 \node[knoten] (53) at (6.35,5.25) {};
 \node[knoten] (54) at (6.25,4.93) {};
 \node[knoten] (61) at (7.10,3.55) {};
 \node[knoten] (62) at (7.20,3.85) {};
 \node[knoten] (63) at (7.50,3.75) {};
 \node[knoten] (64) at (7.40,3.45) {};
 \node[knoten] (71) at (7.03,2.71) {};
 \node[knoten] (72) at (7.13,3.01) {};
 \node[knoten] (73) at (7.43,2.91) {};
 \node[knoten] (74) at (7.33,2.61) {};
 \node[knoten] (81) at (6.60,1.95) {};
 \node[knoten] (82) at (6.70,2.25) {};
 \node[knoten] (83) at (7.00,2.15) {};
 \node[knoten] (84) at (6.90,1.85) {};
 \node[knoten] (91) at (5.45,2.05) {};
 \node[knoten] (92) at (5.55,2.35) {};
 \node[knoten] (93) at (5.85,2.25) {};
 \node[knoten] (94) at (5.75,1.95) {};
 \node[knoten] (101) at (4.95,2.71) {};
 \node[knoten] (102) at (5.05,3.01) {};
 \node[knoten] (103) at (5.35,2.91) {};
 \node[knoten] (104) at (5.25,2.61) {};
 \node[knoten] (111) at (4.80,3.55) {};
 \node[knoten] (112) at (4.90,3.85) {};
 \node[knoten] (113) at (5.20,3.75) {};
 \node[knoten] (114) at (5.10,3.45) {};

 \node[knoten] (161) at (6.95,4.39) {};
 \node[knoten] (162) at (7.05,4.69) {};
 \node[knoten] (163) at (7.35,4.59) {};
 \node[knoten] (164) at (7.25,4.29) {};

 \node[knoten] (151) at (4.95,4.39) {};
 \node[knoten] (152) at (5.05,4.69) {};
 \node[knoten] (153) at (5.35,4.59) {};
 \node[knoten] (154) at (5.25,4.29) {};

 \draw (51) -- (52);
 \draw (51) -- (53);
 \draw (51) -- (54);
 \draw (52) -- (53);
 \draw (52) -- (54);
 \draw (53) -- (54);

 \draw (61) -- (62);
 \draw (61) -- (63);
 \draw (61) -- (64);
 \draw (62) -- (63);
 \draw (62) -- (64);
 \draw (63) -- (64);

 \draw (71) -- (72);
 \draw (71) -- (73);
 \draw (71) -- (74);
 \draw (72) -- (73);
 \draw (72) -- (74);
 \draw (73) -- (74);

 \draw (81) -- (82);
 \draw (81) -- (83);
 \draw (81) -- (84);
 \draw (82) -- (83);
 \draw (82) -- (84);
 \draw (83) -- (84);

 \draw (91) -- (92);
 \draw (91) -- (93);
 \draw (91) -- (94);
 \draw (92) -- (93);
 \draw (92) -- (94);
 \draw (93) -- (94);

 \draw (101) -- (102);
 \draw (101) -- (103);
 \draw (101) -- (104);
 \draw (102) -- (103);
 \draw (102) -- (104);
 \draw (103) -- (104);

 \draw (111) -- (112);
 \draw (111) -- (113);
 \draw (111) -- (114);
 \draw (112) -- (113);
 \draw (112) -- (114);
 \draw (113) -- (114);

 \draw (151) -- (152);
 \draw (151) -- (153);
 \draw (151) -- (154);
 \draw (152) -- (153);
 \draw (152) -- (154);
 \draw (153) -- (154);

 \draw (161) -- (162);
 \draw (161) -- (163);
 \draw (161) -- (164);
 \draw (162) -- (163);
 \draw (162) -- (164);
 \draw (163) -- (164);

 \draw (51) -- (161);
 \draw (52) -- (162);
 \draw (53) -- (163);
 \draw (54) -- (164);
 \draw (61) -- (71);
 \draw (62) -- (72);
 \draw (63) -- (73);
 \draw (64) -- (74);
 \draw (71) -- (81);
 \draw (72) -- (82);
 \draw (73) -- (83);
 \draw (74) -- (84);
 \draw (81) -- (91);
 \draw (82) -- (92);
 \draw (83) -- (93);
 \draw (84) -- (94);
 \draw (91) -- (101);
 \draw (92) -- (102);
 \draw (93) -- (103);
 \draw (94) -- (104);
 \draw (101) -- (111);
 \draw (102) -- (112);
 \draw (103) -- (113);
 \draw (104) -- (114);
 \draw (151) -- (111);
 \draw (152) -- (112);
 \draw (153) -- (113);
 \draw (154) -- (114);
 \draw (151) -- (51);
 \draw (152) -- (52);
 \draw (153) -- (53);
 \draw (154) -- (54);
%


  \draw (61) -- (161);
  \draw (62) -- (162);
  \draw (63) -- (163);
  \draw (64) -- (164);

\end{tikzpicture}

\caption{Illustration of the construction of the graph $G$, where $d-1=4$ and $n/(d-1)=9$, so $n=36$.}\label{fig:graph}
\end{figure}

Let us now consider the number of balls that are born in each clique.
This can be seen as the $1$-choice process where $n$ balls are randomly placed into $n/(d-1)$ bins.
By \citet[Theorem 1, Cases 1 \& 2]{RS98} it follows that, with probability $1-o(1)$, there exists a clique with vertex set $S$,
$|S| = d-1$, so that the number of balls born in $S$ is at least
\[
\Phi_{S} := C \cdot \left( \frac{\log n}{ \log\left( \frac{\log (n/(d-1))}{d-1}  \right)}  \right),
\]
for some constant $C > 0$.
Now we use \lemref{characterizationlower} with the $\Phi_{S}$ above to conclude that the maximum load $\beta = X_{\max}^{(n)} $ satisfies
\[
  \beta \cdot | B_{S}^{\beta} | \geq C \cdot \left( \frac{\log n}{ \log\left( \frac{\log (n/(d-1))}{d-1}  \right)}  \right).
\]
Since $| B_{S}^{\beta}| \leq (2 \beta + 1) \cdot d$ and $\log(n/(d-1)) = \Theta(\log n)$, this implies
$
 \beta = \Omega \left(	\sqrt{ \frac{\log n}{d \cdot \log (\frac{\log n}{d})}			}			\right).
$
%
%
\end{proof}
\bibliographystyle{abbrvnat}
\bibliography{ballbin}

\appendix
\section{Standard technical results}
\begin{lem}[{Azuma's inequality~\cite[Theorem~7.2.1]{AlonSpencer}}]\label{lem:azuma}
   Let $X_0,X_1,\ldots,X_m$ be a martingale such that there exists a fixed positive $c$ for which $|X_i-X_{i-1}|\leq c$ for all $i$.
   Then,
   $$
      \Pro{|X_m - X_0| \geq \lambda} \leq \exp\left(-\frac{\lambda^2}{2c^2m}\right).
   $$
\end{lem}
\begin{lem}[{Azuma's inequality with variance bound~\cite[Theorem~6.1]{CL06}}]\label{lem:azumavar}
   Let $X_0,X_1,\ldots,X_m$ be a martingale adapted to the filtration $\mathcal{F}_i$.
   Suppose that there exists a fixed positive $c$ for which $|X_i-X_{i-1}|\leq c$ for all $i$ and
   there exists $c'$ such that $\Ex{(X_i-X_{i-1})^2 \Mid \mathcal{F}_{i-1}}\leq c'$ for all $i$.
   Then,
   $$
      \Pro{|X_m - X_0| \geq \lambda} \leq \exp\left(-\frac{\lambda^2}{2c'm + c\lambda /3}\right).
   $$
\end{lem}

For the special case where $X_0,X_1,\ldots,X_m$ are independent Bernoulli random variables, we can apply the above lemma 
to the random variables $(X_i-\Ex{X_i})_i$ with $c'=\Ex{X_1}$ and $c=1$ to obtain the inequality below. 
\begin{lem}\label{lem:improvedhoeffding}
   Let $X_1,\ldots,X_m$ be $m$ independent, identically distributed Bernoulli random variables. 
   Let $X:=\sum_{i=1}^m X_i$. Then, for any $\lambda > 0$,
\begin{align*}
 \Pro{ | X - \Ex{X} | \geq \lambda } \leq \exp \left(- \frac{\lambda^2}{ 2\Ex{X} + \lambda/3 }		 \right).
\end{align*}
\end{lem}

\begin{lem}\label{lem:thomasrepair}
Consider the $1$-choice process $\{ \overline{X}_v^{(n)} \}_{v \in V}$ where $n$ balls are allocated to $n$ bins chosen independently and uniformly at random.
Let  $\ell_0 := 8 \ce \Delta^2$ and let $\Lambda_{\ell} := \left\{ u \in V \colon X_{u}^{(n)} \geq \ell \right \}$.
Then, for any $\ell$ with $\ell_0 \leq \ell = o(\log^2 n)$,
\begin{align*}
 \Pro{  | \Lambda_{\ell} | \geq \frac{n}{4 \Delta} \cdot (2 \Delta)^{-\ell} + \frac{\log^7 n}{\ell}  } \leq 2 n^{-\log^5 n}.
\end{align*}
\end{lem}
\begin{proof}
   Fix any $\ell$ with $\ell \geq \ell_0$. Let $\{ \tilde{X}_{v} \}_{v \in V}$ be $n$ independent poisson random variables with mean $1$.
   Define $\tilde{\Lambda}_{\ell} := \left\{ u \in V \colon \tilde{X}_{u}^{(n)} \geq \ell \right \}$.
   Since $\{ |\tilde{\Lambda}_{\ell}| \geq \frac{|V|}{4 \Delta} \cdot (2 \Delta)^{-\ell} + \frac{\log^7 |V|}{\ell} \}$ is a monotone event in the number
   of balls $1 \leq n \leq |V|$, it follows by a standard ``Poissonization argument'' (see, e.g., \cite[Corollary 5.11]{MU05})
   \begin{align*}
    \Pro{ | \Lambda_{\ell} | \geq \frac{n}{4 \Delta} \cdot (2 \Delta)^{-\ell} + \frac{\log^7 n}{\ell} } &\leq 2 \cdot   \Pro{ \tilde{\Lambda}_{\ell} \geq \frac{n}{4 \Delta} \cdot (2 \Delta)^{-\ell} + \frac{\log^7 n}{\ell} }.
   \end{align*}
   To bound the latter probability, let us first estimate $\Ex{ |\tilde{\Lambda}_{\ell}| }$.
   First, if $P$ is a Poisson random variable with parameter $1$, then we have the following Chernoff-type inequality
   (\cite[Theorem~A.1.15]{AlonSpencer}): for any $\epsilon > 0$,
   \begin{align*}
    \Pro{ P \geq (1 + \epsilon) } &\leq 	\ce^{\epsilon} (1+\epsilon)^{-(1+\epsilon)}.
   \end{align*}
   As long as $\epsilon \geq 8 \ce \Delta^2 - 1 $, we can write
   \begin{align*}
    \Pro{ P \geq 1 + \epsilon } &\leq 	\ce^{1+\epsilon} (8 \ce \Delta^2)^{-(1+\epsilon)} \leq (8 \Delta^2)^{-(1+\epsilon)},
   \end{align*}
   and hence, replacing $1+\epsilon$ by $\ell$ gives
   \begin{align*}
    \Pro{ P \geq \ell} &\leq (8 \Delta^2)^{-\ell}.
   \end{align*}
   Now observe that $|\tilde{\Lambda}_{\ell}|$ is stochastically smaller than the sum of $n$ independent
   Bernoulli random variables $Z_1,\ldots,Z_n$, each with parameter $(8 \Delta^2)^{-\ell}$. Hence, if we denote $Z:=\sum_{i=1}^n Z_i$, then
   \begin{align*}
    \Pro{ | \tilde{\Lambda}_{\ell}| \geq \frac{n}{4 \Delta} \cdot (2 \Delta)^{-\ell} + \frac{\log^7 n}{\ell} }
    &\leq  \Pro{ Z \geq \frac{n}{4 \Delta} \cdot (2 \Delta)^{-\ell} + \frac{\log^7 n}{\ell} }.
   \end{align*}
   Note that $\Ex{Z} = n \cdot (8 \Delta^2)^{-\ell}$. Hence by \lemref{improvedhoeffding},
   \begin{align*}
   \Pro{ Z \geq \frac{n}{4 \Delta} \cdot (2 \Delta)^{-\ell} + \frac{\log^7 n}{\ell} } &\leq \Pro{ Z \geq \Ex{Z} + \frac{\log^7 n}{\ell} + \frac{n}{8 \Delta} \cdot (2 \Delta)^{-\ell} }
   \\ &\leq \exp \left(-	\frac{  \left( \frac{\log^7 n}{\ell} + \frac{n}{8 \Delta} \cdot (2 \Delta)^{-\ell}  \right)^2 }{2 \Ex{Z} +  \frac{\log^7 n}{\ell} + \frac{n}{8 \Delta} \cdot (2 \Delta)^{-\ell}   }		\right).
   \end{align*}
   To bound the last term, we proceed by a case distinction. The first case is
   when $\log^7 n / \ell \geq \frac{n}{8 \Delta} \cdot (2 \Delta)^{-\ell}$. Then also $\Ex{Z} \leq \log^7 n / \ell$ and hence
   \begin{align*}
   \exp \left(-	\frac{  \left( \frac{\log^7 n}{\ell} + \frac{n}{8 \Delta} \cdot (2 \Delta)^{-\ell}  \right)^2 }{2\Ex{Z} +  \frac{\log^7 n}{\ell} + \frac{n}{8 \Delta} \cdot (2 \Delta)^{-\ell}}		\right) &\leq
   \exp \left( - \frac{ \left(\frac{\log^7 n}{\ell} \right)^2    }{ 4 \cdot \frac{\log^7 n}{\ell}  }
   \right) \leq n^{-\log^5 n}.
   \end{align*}
   Otherwise, $\log^7 n / \ell < \frac{n}{8 \Delta} \cdot (2 \Delta)^{-\ell}$. Then,
   \begin{align*}
   \exp \left(-	\frac{  \left( \frac{\log^7 n}{\ell} + \frac{n}{8 \Delta} \cdot (2 \Delta)^{-\ell}  \right)^2 }{2 \Ex{Z} +  \frac{\log^7 n}{\ell} + \frac{n}{8 \Delta} \cdot (2 \Delta)^{-\ell}}		\right) &\leq
   \exp \left( - \frac{ \left( \frac{n}{8 \Delta} \cdot (2 \Delta)^{-\ell} \right)^2    }{ 4 \cdot \frac{n}{8 \Delta} \cdot (2 \Delta)^{-\ell}  }
   \right) \leq n^{-\log^5 n}.
   \end{align*}

   Hence,
   \begin{align*}
   \Pro{ | \Lambda_{\ell} | \geq \frac{n}{4 \Delta} \cdot (2 \Delta)^{-\ell} + \frac{\log^7 n}{\ell} } &\leq 2 \cdot n^{-\log^5 n}.
   \end{align*}
   as desired.
\end{proof}

\end{document}